\documentclass[journal]{IEEEtran}
\usepackage{amsmath,amssymb,mathtools,amsthm}
\usepackage{graphicx}
\usepackage{cite}
\usepackage{placeins}
\usepackage{tikz}
\usetikzlibrary{arrows.meta}

\newtheorem{theorem}{Theorem}
\newtheorem{lemma}{Lemma}
\newtheorem{proposition}{Proposition}
\newtheorem{corollary}{Corollary}
\theoremstyle{definition}
\newtheorem{assumption}{Assumption}
\newtheorem{definition}{Definition}
\theoremstyle{remark}
\newtheorem{remark}{Remark}

\newcommand{\R}{\mathbb{R}}
\newcommand{\ones}{\mathbf{1}}
\newcommand{\range}{\operatorname{range}}
\newcommand{\rank}{\operatorname{rank}}

\newcommand{\REV}[1]{{\color{black}#1}}
\RequirePackage{color}\definecolor{RED}{rgb}{1,0,0}\definecolor{BLUE}{rgb}{0,0,0} 
\DeclareOldFontCommand{\sf}{\normalfont\sffamily}{\mathsf} 
\providecommand{\DIFadd}[1]{{\color{black}#1}} 
\providecommand{\DIFdel}[1]{} 
\providecommand{\DIFaddbegin}{\color{black}} 
\providecommand{\DIFaddend}{\color{black}} 
\providecommand{\DIFdelbegin}{} 
\providecommand{\DIFdelend}{} 
\providecommand{\DIFdelFL}[1]{} 
\providecommand{\DIFaddbeginFL}{\color{black}} 
\providecommand{\DIFaddendFL}{\color{black}} 
\providecommand{\DIFdelbeginFL}{} 
\providecommand{\DIFdelendFL}{} 
\newcommand{\DIFscaledelfig}{0.5}
\RequirePackage{settobox} 
\RequirePackage{letltxmacro} 
\newsavebox{\DIFdelgraphicsbox} 
\newlength{\DIFdelgraphicswidth} 
\newlength{\DIFdelgraphicsheight} 
\LetLtxMacro{\DIFOincludegraphics}{\includegraphics} 
\newcommand{\DIFaddincludegraphics}[2][]{{\color{black}\fbox{\DIFOincludegraphics[#1]{#2}}}} 
\newcommand{\DIFdelincludegraphics}[2][]{
\sbox{\DIFdelgraphicsbox}{\DIFOincludegraphics[#1]{#2}}
\settoboxwidth{\DIFdelgraphicswidth}{\DIFdelgraphicsbox} 
\settoboxtotalheight{\DIFdelgraphicsheight}{\DIFdelgraphicsbox} 
\scalebox{\DIFscaledelfig}{
\parbox[b]{\DIFdelgraphicswidth}{\usebox{\DIFdelgraphicsbox}\\[-\baselineskip] \rule{\DIFdelgraphicswidth}{0em}}\llap{\resizebox{\DIFdelgraphicswidth}{\DIFdelgraphicsheight}{
\setlength{\unitlength}{\DIFdelgraphicswidth}
\begin{picture}(1,1)
\thicklines\linethickness{2pt} 
{\color[rgb]{1,0,0}\put(0,0){\framebox(1,1){}}}
{\color[rgb]{1,0,0}\put(0,0){\line( 1,1){1}}}
{\color[rgb]{1,0,0}\put(0,1){\line(1,-1){1}}}
\end{picture}
}\hspace*{3pt}}} 
} 
\LetLtxMacro{\DIFOaddbegin}{\DIFaddbegin} 
\LetLtxMacro{\DIFOaddend}{\DIFaddend} 
\LetLtxMacro{\DIFOdelbegin}{\DIFdelbegin} 
\LetLtxMacro{\DIFOdelend}{\DIFdelend} 
\DeclareRobustCommand{\DIFaddbegin}{\DIFOaddbegin \let\includegraphics\DIFaddincludegraphics} 
\DeclareRobustCommand{\DIFaddend}{\DIFOaddend \let\includegraphics\DIFOincludegraphics} 
\DeclareRobustCommand{\DIFdelbegin}{\DIFOdelbegin \let\includegraphics\DIFdelincludegraphics} 
\DeclareRobustCommand{\DIFdelend}{\DIFOaddend \let\includegraphics\DIFOincludegraphics} 
\LetLtxMacro{\DIFOaddbeginFL}{\DIFaddbeginFL} 
\LetLtxMacro{\DIFOaddendFL}{\DIFaddendFL} 
\LetLtxMacro{\DIFOdelbeginFL}{\DIFdelbeginFL} 
\LetLtxMacro{\DIFOdelendFL}{\DIFdelendFL} 
\DeclareRobustCommand{\DIFaddbeginFL}{\DIFOaddbeginFL \let\includegraphics\DIFaddincludegraphics} 
\DeclareRobustCommand{\DIFaddendFL}{\DIFOaddendFL \let\includegraphics\DIFOincludegraphics} 
\DeclareRobustCommand{\DIFdelbeginFL}{\DIFOdelbeginFL \let\includegraphics\DIFdelincludegraphics} 
\DeclareRobustCommand{\DIFdelendFL}{\DIFOaddendFL \let\includegraphics\DIFOincludegraphics} 
\RequirePackage{listings} 
\RequirePackage{color} 
\lstdefinelanguage{DIFcode}{ 
  moredelim=[il][\color{red}\scriptsize]{\%DIF\ <\ }, 
  moredelim=[il][\color{black}\sffamily]{\%DIF\ >\ } 
} 
\lstdefinestyle{DIFverbatimstyle}{ 
	language=DIFcode, 
	basicstyle=\ttfamily, 
	columns=fullflexible, 
	keepspaces=true 
} 
\lstnewenvironment{DIFverbatim}{\lstset{style=DIFverbatimstyle}}{} 
\lstnewenvironment{DIFverbatim*}{\lstset{style=DIFverbatimstyle,showspaces=true}}{} 
\begin{document}

\title{Sparse Feedback Implementation for Sender--Receiver Transportation Linear-Quadratic Control}

\author{Anders Hansson and Johan L{\"o}fberg\thanks{Anders Hansson and Johan L{\"o}fberg are with the Division of Automatic Control, Link\"oping University, SE--581 83 Link\"oping, Sweden. Email: anders.g.hansson@liu.se.}}

\maketitle

\begin{abstract}
We study a sparse linear-quadratic problem for transportation dynamics.  \DIFaddbegin \DIFadd{The sparsity pattern has a natural directed-graph representation in which vertices are storage locations and edges are transportation links.  }\DIFaddend The goal is to compute the optimal control signal without applying the usually dense optimal feedback gain directly.  We show that the optimal feedback gain can be factorized as the product of a sparse matrix and the inverse of another sparse matrix from the right.  \DIFaddbegin \DIFadd{In contrast to an existing factorization that uses an inverse from the left, the proposed factors use graph-adapted state coordinates.  When the underlying graph is a tree, we give, for every edge orientation, a closed formula for the Riccati matrix that determines these factors and a graph-based bound on their possible nonzero entries.  }\DIFaddend This factorization can reduce the online cost relative to dense feedback multiplication and \REV{admits tree-based evaluation}.
The main message is that linear quadratic control need not appear dense when expressed in graph-adapted coordinates, and that these coordinates can reveal intrinsic sparsity.
\end{abstract}


\section{Introduction}

This paper studies sparse implementation of the optimal feedback for a sparse linear-quadratic (LQ) problem with transportation dynamics.  Transportation dynamics describe systems in which material, water, heat, or goods are stored at locations and moved between locations with transport delays.  \REV{Motivating settings} include multi-echelon inventory systems with lead times~\cite{IgnaciukBartoszewicz2010}, structured transportation linear-quadratic problems~\cite{HeydenPatesRantzer2018,HeydenPatesRantzer2020}, irrigation and water-delivery canals~\cite{Malaterre1998,LemosPintoRijoRato2013}, and district-heating networks with flow-induced delays~\cite{MadsenHolst1992,BendtsenValKallesoeKrstic2017}.

\DIFaddbegin \DIFadd{The transportation equations define a directed graph: each storage location is a vertex, and each transportation link is an edge directed from its sender to its receiver.  The graph identifies which storage and transport variables are coupled and which entries of the feedback factors may be nonzero.
}

\DIFaddend %

The applications that motivate this work are large and sparse.  For such systems, computing a centralized optimal feedback gain is not enough: the gain is often dense, and applying it directly can destroy the computational and communication advantages of the underlying network.  We therefore seek an equivalent sparse implementation of the same unconstrained optimal controller.

The closest related work is the recent preprint of Adlercreutz and Pates~\cite{AdlercreutzPates2026}, which develops sparse implementations of LQ controllers for transportation problems based on Cholesky factorizations of shift-operator matrices. Their approach, which we refer to as a left-inverse (LI) implementation, \REV{multiplies the state by a sparse matrix and then solves a sparse system for the input variables} when the underlying network is a tree.  \DIFaddbegin \DIFadd{In contrast, the right-inverse (RI) implementation developed here solves for state coordinates associated with the graph and then applies a sparse input matrix.  \REV{The RI solve factor is fixed by the graph and can be applied by edge sums without numerical factorization.}  }\DIFaddend More broadly, structured control has been studied by imposing communication or sparsity constraints at the design stage, leading to convex formulations under quadratic invariance (QI)~\cite{RotkowitzLall2006} and to system-level synthesis (SLS) approaches that enforce locality directly on the closed-loop responses~\cite{WangMatniDoyle2018,WangMatniDoyle2019}. 
QI preserves convexity under sparsity constraints, and SLS imposes locality on closed-loop responses, both restricting the controller class relative to unconstrained LQ.

\DIFdelbegin \DIFdel{In contrast, we start from the unconstrained LQ solution and extract a sparse implementation via a right-inverse (RI) factorization, showing sparsity is intrinsic to the optimal controller. It is strongly influenced by the underlying graph structure, which provides a priori information about the sparsity pattern. In particular, aspects of the sparsity can be inferred directly from the graph without solving the LQ problem.
}
\DIFaddbegin \DIFadd{For an unconstrained infinite-horizon discrete-time LQ problem, the optimal feedback gain is determined by the stabilizing solution of the discrete-time algebraic Riccati equation (DARE).  We therefore determine this solution and the possible nonzero entries of the resulting feedback factors.
}

\DIFaddend %
\DIFdel{Our route is based on deflating subspaces.  Deflating-subspace methods are a standard way to solve }\DIFdelend \DIFaddbegin \DIFadd{We first reduce the DARE and eliminate the }\DIFaddend \DIFdelbegin \DIFdel{discrete-time }\DIFdelend \DIFdelbegin \DIFdel{linear-quadratic problems through the discrete-time }\DIFdelend \DIFdelbegin \DIFdel{algebraic Riccati equation (DARE) }\DIFdelend \DIFdelbegin \DIFdel{: the Schur-vector method of Laub~\cite{Laub1979}, the generalized-eigenvalue formulation of van Dooren~\cite{VanDooren1981}, and standard Riccati accounts~\cite{BittantiLaubWillems1991} all connect the stabilizing solution to an invariant or deflating subspace.  These methods are reliable, and they are central in numerical Riccati theory, including large-scale and sparse settings~\cite{MehrmannWatkins2000,BennerLiPenzl2008}.  However, the basis returned by a standard deflating-subspace computation is usually dense, even when the control problem is sparse. The key finding of this paper is that the relevant deflating subspace admits a sparse nonorthogonal basis that yields a sparse right-inverse implementation of the optimal feedback.
The existence of such a basis is established under a graph-induced structural constraint that can be verified a priori.
All sparsity guarantees in this paper are derived under a provable structural characterization of DARE solution. We then show that, in all examples and tested cases, a finer graph-interpretable structure appears, which we describe in terms of intersection components.
}\DIFdelend \DIFaddbegin \DIFadd{blocks associated with transportation delays.  \REV{For every edge orientation of a connected tree, we then give an explicit formula for the remaining Riccati matrix.  This formula determines the RI factors and a graph-based bound on their possible nonzero entries.}
}\DIFaddend 

The contributions are \DIFdelbegin \DIFdel{not merely a restatement of the problem setup; they advance the sparse Riccati and transportation LQ control literature in the following ways}\DIFdelend \DIFaddbegin \DIFadd{as follows}\DIFaddend .
\begin{itemize}
    \item \DIFaddbegin \DIFadd{We express the unconstrained optimal feedback as a sparse RI factorization instead of applying the dense feedback gain directly.}\DIFaddend
    \item \DIFaddbegin \DIFadd{For every edge orientation of a connected tree, we derive a closed formula for the reduced DARE solution that determines the RI factors.}\DIFaddend
    \item \DIFaddbegin \DIFadd{We use this formula to bound the possible nonzero entries of the sparse RI input matrix directly from the graph. \REV{These guaranteed zeros may also inform design with an input penalty.}}\DIFaddend
    \item \DIFaddbegin \DIFadd{We report RI factor counts, graph support bounds, and \REV{online operation counts}.}\DIFaddend
\end{itemize}

\paragraph*{\REV{Notation}}
\REV{The following conventions are used to state the network model and graph results.}
Subscripts are used for time indices, labels, and vector and matrix entries.   The dimensions of standard vectors and matrices are inferred from context.  A vector may be written as $x=(x_1,\ldots,x_n)$.  In matrix products, this vector is tacitly identified with the column matrix with entries $x_1,\ldots,x_n$.  For a statement $A$, the indicator $1_A$ equals $1$ if $A$ is true and equals $0$ otherwise.  For a finite set $W$ with a fixed ordering and a subset $\Omega\subseteq W$, the indicator vector of $\Omega$ in $\R^{|W|}$ is the vector whose entry indexed by $w\in W$ is $1_{w\in\Omega}$.  The indicator vector of a tree means the indicator vector of its vertex set.  Matrix row or column indices fix the ordering when $W$ is an index set.  The standard basis vector indexed by $i$ is denoted $\hat e_i$.  An undirected path is a list of distinct vertices in which consecutive vertices are joined by an edge, ignoring edge orientation.  A chain for a binary relation is a finite list of indices in which each consecutive pair satisfies the relation.  An equivalence class of a relation $\sim$ is a maximal set of indices such that $p\sim q$ for every pair $p,q$ in the set.  Symbols tied to the graph are introduced where they first appear.

\section{Transportation linear-quadratic problem}\label{sec:problem}

\REV{We specify the network, its optimal-control objective, and the graph assumptions used in the subsequent results.}

Let $V$ be the set of storage vertices and let $E$ be the set of directed transportation edges.  We write $G=(V,E)$ for this directed graph, with $|V|=N$ and $|E|=m$.  Each vertex corresponds to one storage state, and each edge corresponds to one delayed transport state and one input channel.  The edge orientation specifies sending and receiving endpoints.  For each edge, let its sender be the vertex from which the current input removes material and let its receiver be the vertex to which the delayed edge state later adds material.  Enumerate the edges as $E=\{e_1,\ldots,e_m\}$.  Define the sender and receiver selection matrices $S,R\in\R^{N\times m}$ by taking the $i$th column of $S$ to be the standard basis vector at the sender of edge $e_i$, and the $i$th column of $R$ to be the standard basis vector at the receiver of edge $e_i$.  The signed incidence matrix is $C=R-S\in\R^{N\times m}$.  Thus the $i$th column of $C$ has one entry $+1$ at the receiver of $e_i$, one entry $-1$ at the sender of $e_i$, and zeros elsewhere.  Reversing an edge only changes the sign of one column and interchanges its sender and receiver.

\begingroup\color{black}
At time $k$, let $\widetilde x_k=(\widetilde s_k,\widetilde d_k)\in\R^{N+m}$,
where $\widetilde s_k\in\R^N$ are storage states and
$\widetilde d_k\in\R^m$ are delayed transport states. The input is
$\widetilde u_k\in\R^m$. For a discount factor $0<\beta<1$, the optimal control
problem is
\begin{equation*}
\begin{aligned}
\text{minimize}\quad &\sum_{k=0}^{\infty}\beta^k\widetilde s_k^\top\widetilde s_k\\
\text{subject to}\quad
&\widetilde s_{k+1}=\widetilde s_k+R\widetilde d_k-S\widetilde u_k,\\
&\widetilde d_{k+1}=\widetilde u_k,\qquad k\ge0.
\end{aligned}
\end{equation*}
The optimization variables are $\widetilde x_k$ and $\widetilde u_k$;
$\widetilde x_0$ is given. States and inputs are unrestricted real variables.
This objective has no input penalty. Each input acts immediately at its
sender and after one time-step delay at its receiver, with the signs of
both contributions determined by the input sign.

To use the standard undiscounted LQ formulation and its DARE, we absorb the
discount into the state and input. Set $r=\sqrt\beta$,
$x_k=r^k\widetilde x_k=(s_k,d_k)$, and $u_k=r^{k+1}\widetilde u_k$.
The resulting sender--receiver dynamics are
\endgroup
\begin{equation}\label{eq:sr-dynamics}
    x_{k+1}=Ax_k+Bu_k,
    \quad
    A=r\begin{bmatrix}I&R\\0&0\end{bmatrix},
    \quad
    B=\begin{bmatrix}-S\\I_m\end{bmatrix}.
\end{equation}
\begingroup\color{black}
In these coordinates,
$\beta^k\widetilde s_k^\top\widetilde s_k=s_k^\top s_k=x_k^\top Qx_k$,
where $Q=\left[\begin{smallmatrix}I_N&0\\0&0\end{smallmatrix}\right]$.
The equivalent undiscounted problem is
\begin{equation}\label{eq:control-problem}
\begin{aligned}
\text{minimize}\quad &\sum_{k=0}^{\infty}x_k^\top Qx_k\\
\text{subject to}\quad &\eqref{eq:sr-dynamics}.
\end{aligned}
\end{equation}
The optimization variables are $x_k$ and $u_k$; $x_0$ is given.
\endgroup

\REV{When additional application-specific constraints are imposed on
\eqref{eq:control-problem}, its unconstrained optimum provides a lower bound
on the achievable cost, serving as a performance-limit
benchmark~\cite{BoydBarratt1991}. The same benchmarking reasoning applies
when a nonnegative input penalty is added.}

\begin{assumption}\label{ass:data}
The graph $G$ is a connected directed tree.
\end{assumption}

No restriction is placed on the edge orientation.  In particular, several directed edges may have the same receiving vertex.

\DIFaddbegin \DIFadd{Assumption~\ref{ass:data} describes radial transportation networks, such as distribution trees and \REV{tree-structured} supply chains.  It excludes cycles, which require incidence coordinates and arguments different from those used for trees.
}

\DIFadd{Proofs of all formal statements are collected at the end of the paper.
}

\DIFaddend \begin{lemma}\label{lem:tree-incidence-rank-Z}
Under Assumption~\ref{ass:data}, $m=N-1$, $\rank C=N-1$, and the matrix
\begin{equation}\label{eq:Z-definition}
    Z=[C\ \ones]\in\R^{N\times N}
\end{equation}
is nonsingular.
\end{lemma}

\REV{For a specialized support bound on the RI input matrix, we use the following additional rank condition.}

\begin{assumption}\label{ass:outward-rooted}
In addition to Assumption~\ref{ass:data}, the receiver matrix has full column rank, $\rank R=m$.
\end{assumption}

\begin{lemma}\label{lem:outward-rooted-orientation}
Under Assumption~\ref{ass:outward-rooted}, no two directed edges have the same receiving vertex.  There is a unique vertex $v_{\rm root}$ with no incoming edge, and every edge is directed away from $v_{\rm root}$.
\end{lemma}

\begin{definition}\label{def:rooted-tree-notation}
Suppose the transportation graph $G=(V,E)$ satisfies Assumption~\ref{ass:outward-rooted}, and let $v_{\rm root}$ be the vertex from Lemma~\ref{lem:outward-rooted-orientation}.  For each vertex $v\ne v_{\rm root}$, let $p(v)$ be its parent.  Let $T_v$ be the subtree rooted at $v$, consisting of $v$ and all its descendants, and let $\chi_v\in\R^N$ be the indicator vector of $T_v$.
\end{definition}
\section{DARE\DIFdelbegin \DIFdel{and generalized eigenvalue pencil}\DIFdelend }\label{sec:riccati-pencil}

\REV{We first establish the stabilizing solution of the full DARE; the next section reduces it to the storage-state dimension and constructs the RI factors.}

For matrices satisfying Assumption~\ref{ass:data}, consider the DARE
\begin{equation}\label{eq:dare}
    P
    =
    Q+A^\top P A
    -A^\top P B\,(B^\top P B)^{-1}B^\top P A.
\end{equation}
A symmetric positive semidefinite solution $P$ is called stabilizing if $B^\top P B$ is nonsingular and the closed-loop matrix \(A-BK, \qquad K=(B^\top P B)^{-1}B^\top P A\),
has spectral radius strictly smaller than one.

\begin{theorem}\label{thm:sender-receiver-dare-unique}
Under Assumption~\ref{ass:data}, the DARE \eqref{eq:dare} has a unique positive semidefinite stabilizing solution.
\end{theorem}

\DIFdelbegin \DIFdel{The generalized eigenvalue problem associated with \eqref{eq:dare} is the extended DARE pencil
}\DIFdel{
    \mathcal H-\lambda \mathcal J,
}
\DIFdel{where
}
\DIFdel{\mathcal H=\begin{bmatrix}
        A&0&B\\
        -Q&I&0\\
        0&0&0
    \end{bmatrix},
    \qquad
\mathcal J=\begin{bmatrix}
        I&0&B\\
        0&A^\top&0\\
        0&-B^\top&0
    \end{bmatrix}.
}
\DIFdel{Here the row and column blocks have sizes $(n,n,m)$ with $n=N+m$.  A matrix $W$ spans a right deflating subspace of the pencil if there is a reduced matrix $L$ such that \(\mathcal H W=\mathcal J W L\).
Such a deflating subspace is called stable if all generalized eigenvalues represented by the reduced matrix $L$ lie strictly inside the unit disk, equivalently if $\rho(L)<1$.
Write a stable deflating-subspace basis in block form as
}
    \DIFdel{W=\begin{bmatrix}X\\Y\\U\end{bmatrix},
    \qquad X,Y\in\R^{n\times n},\quad U\in\R^{m\times n}.
}
\DIFdel{The direction used in this paper is the following standard construction.  If $P$ is a stabilizing solution of \eqref{eq:dare} and $Y\in\R^{n\times n}$ is any nonsingular matrix, then, with \(X=PY, \qquad U=-KY\),
the columns of $W=[X^\top\ Y^\top\ U^\top]^\top$ span the stabilizing deflating subspace of the pencil \eqref{eq:pencil}.  In this sense, the denominator block $Y$ may be chosen after the stabilizing Riccati solution is known.
Below, for the sender--receiver problem, we choose a graph-adapted block $Y$ and then record that this choice is nonsingular and gives the corresponding stabilizing deflating-subspace basis.
}\DIFdelend 

\DIFdelbegin \DIFdel{Conversely, recovering a stabilizing DARE solution from a computed deflating subspace requires additional regularity, in particular a nonsingular denominator block and the usual stabilizing spectral separation assumptions.  Under those hypotheses, the first two block rows satisfy \(X=PY\)
and the solution of \eqref{eq:dare} is recovered as \(P=XY^{-1}\).
With the convention $u_k=-Kx_k$, the input block satisfies $U=-KY$, so the feedback is recovered from the same basis by \(K=-UY^{-1}\).
This converse subspace-to-Riccati correspondence is standard in generalized-eigenvalue treatments of DAREs; see, for example, van Dooren~\cite{VanDooren1981} or the account in \cite{BittantiLaubWillems1991}.  The generalized Schur method computes an orthogonal basis for this stable deflating subspace; the central question here is whether the same subspace has a sparse nonorthogonal basis $W$ adapted to the graph.
}

\DIFdelend \section{Sender--receiver Riccati reduction}\label{sec:general-B-framework}

\REV{We reduce the full DARE to a storage-state equation, then construct an exact RI implementation conditional on its solution.  The subsequent tree analysis supplies that solution explicitly and bounds the input-factor support.}

\DIFdelbegin \DIFdel{Write a candidate Riccati matrix in storage-delay block form as
}\DIFdelend \DIFaddbegin \DIFadd{Conformably with the storage-delay partition of the state in \eqref{eq:sr-dynamics}, write a candidate Riccati matrix as
}\DIFaddend %
\begin{equation}\label{eq:P-block-partition}
    P=\begin{bmatrix}P_1&P_{12}\\P_{12}^\top&P_2\end{bmatrix}.
\end{equation}
For the sender--receiver data in \eqref{eq:sr-dynamics}, the denominator in \eqref{eq:dare} is
\begin{equation}\label{eq:general-B-denominator}
\begin{aligned}
    B^\top P B
    &=S^\top P_1 S-S^\top P_{12}-P_{12}^\top S+P_2.
\end{aligned}
\end{equation}
The numerator in the feedback is
\[
\begin{aligned}
    B^\top P A
    &=r\begin{bmatrix}
        -S^\top P_1+P_{12}^\top & -S^\top P_1 R+P_{12}^\top R
    \end{bmatrix}.
\end{aligned}
\]

\subsection{A direct Riccati reduction}\label{subsec:direct-riccati-reduction}

\DIFdelbegin \DIFdel{Constructing the sender--receiver deflating-subspace basis directly is delicate.  We therefore }\DIFdelend \DIFaddbegin \DIFadd{We }\DIFaddend first reduce the DARE \eqref{eq:dare}, and \DIFdelbegin \DIFdel{only later }\DIFdelend \DIFaddbegin \DIFadd{then }\DIFaddend ask whether the resulting \DIFdelbegin \DIFdel{solution }\DIFdelend \DIFaddbegin \DIFadd{optimal feedback }\DIFaddend admits a sparse \DIFdelbegin \DIFdel{denominator basis}\DIFdelend \DIFaddbegin \DIFadd{RI implementation}\DIFaddend .

Using the partitioning in \eqref{eq:P-block-partition}, define
\begin{align}
    D(P)&=S^\top P_1S-S^\top P_{12}-P_{12}^\top S+P_2,
        \label{eq:sr-DP}\\
    N(P)&=P_{12}-P_1S,
        \label{eq:sr-NP}\\
    H(P)&=P_1-N(P)D(P)^{-1}N(P)^\top,
        \label{eq:sr-H-of-P}
\end{align}
whenever $D(P)$ is nonsingular.
Using \eqref{eq:dare} and the definitions \eqref{eq:sr-DP}--\eqref{eq:sr-H-of-P}, the DARE \eqref{eq:dare} is equivalent to
\begin{equation}\label{eq:sr-P-from-H}
    P=\begin{bmatrix}I&0\\0&0\end{bmatrix}
      +\beta
      \begin{bmatrix}I\\R^\top\end{bmatrix}
      H(P)
      \begin{bmatrix}I&R\end{bmatrix}.
\end{equation}
Equating the blocks in \eqref{eq:sr-P-from-H} shows that any solution must have the block form
\begin{equation}\label{eq:sr-P-block-form}
    P_1=I+\beta H,
    \qquad
    P_{12}=\beta HR,
    \qquad
    P_2=\beta R^\top HR,
\end{equation}
for some symmetric matrix $H\in\R^{N\times N}$.
For this matrix $H$, define the two auxiliary matrices
\begin{align}
    D&=S^\top S+\beta C^\top H C,
        \label{eq:sr-DH}\\
    N&=-S+\beta HC.
        \label{eq:sr-NH}
\end{align}
Substituting \eqref{eq:sr-P-block-form} into \eqref{eq:sr-DP}--\eqref{eq:sr-H-of-P} and using $C=R-S$ gives $D(P)=D$, $N(P)=N$, and $H(P)=I+\beta H-ND^{-1}N^\top$.  Equating $H(P)$ and $H$ gives
\begin{equation}\label{eq:sr-reduced-H-DARE}
    H=I+\beta H-ND^{-1}N^\top,
\end{equation}
or, after substituting \eqref{eq:sr-DH} and \eqref{eq:sr-NH} into \eqref{eq:sr-reduced-H-DARE}, equivalently
{\small
\begin{equation}\label{eq:sr-reduced-H-DARE-expanded}
\begin{aligned}
    H={}&I+\beta H 
      -(-S+\beta HC)\\
      &\quad\times(S^\top S+\beta C^\top H C)^{-1}
        (-S^\top+\beta C^\top H).
\end{aligned}
\end{equation}
}
We call \eqref{eq:sr-reduced-H-DARE}--\eqref{eq:sr-reduced-H-DARE-expanded} the reduced DARE.
For a solution $H$ of the reduced DARE with $D$ in \eqref{eq:sr-DH} nonsingular, define the reduced closed-loop matrix
\begin{equation}\label{eq:sr-reduced-closed-loop}
    A_H=r\bigl(I-CD^{-1}N^\top\bigr)
       =rI-rCD^{-1}N^\top,
\end{equation}
and call $H$ stabilizing if every eigenvalue of \eqref{eq:sr-reduced-closed-loop} lies in the open unit disk.

\begin{lemma}\label{lem:full-reduced-dare-equivalence}
Let $H=H^\top$ and suppose $D$ in \eqref{eq:sr-DH} is nonsingular.  Define $P$ by \eqref{eq:sr-P-block-form}.  Then $H$ satisfies the reduced DARE \eqref{eq:sr-reduced-H-DARE-expanded} if and only if $P$ satisfies the original DARE \eqref{eq:dare}.  Moreover, $H$ is stabilizing if and only if $P$ is stabilizing for \eqref{eq:dare}.
\end{lemma}

Using \eqref{eq:dare}, \eqref{eq:sr-DH}, \eqref{eq:sr-NH}, and Lemma~\ref{lem:full-reduced-dare-equivalence}, the feedback associated with a stabilizing solution $H$ of the reduced DARE is
\[
    u_k=-Kx_k,
    \qquad
    K=rD^{-1}N^\top\begin{bmatrix}I&R\end{bmatrix},
\]
with the sign convention of \eqref{eq:dare}.  Equivalently,
\begin{equation}\label{eq:sr-feedback-from-H-state}
    u_k=-rD^{-1}(-S^\top+\beta C^\top H)(s_k+Rd_k).
\end{equation}

This reduction is useful for two reasons.  First, the unknown matrix is $H\in\R^{N\times N}$ rather than the full $(N+m)\times(N+m)$ matrix $P$, whose blocks are fixed by the sender--receiver dynamics.  Second, all sender--receiver dependence enters through $S^\top S$ and $C^\top H C$.  Lemma~\ref{lem:full-reduced-dare-equivalence} shows that the stabilizing reduced solution recovers the stabilizing solution of the original DARE through \eqref{eq:sr-P-block-form}.  \REV{We next establish existence of the reduced solution and a nonsingular denominator.}

\begin{lemma}\label{lem:reduced-dare-existence-conditions}
Under Assumption~\ref{ass:data}, the reduced DARE \eqref{eq:sr-reduced-H-DARE-expanded} has a unique stabilizing positive semidefinite solution $H$.  For this solution, the matrix $D$ in \eqref{eq:sr-DH} is nonsingular.
\end{lemma}

\subsection{An RI \DIFdelbegin \DIFdel{basis }\DIFdelend \DIFaddbegin \DIFadd{implementation }\DIFaddend after $H$ is known}\label{subsec:known-H-denominator-basis}

\REV{We now construct the exact RI pair; its input-factor support is bounded after the explicit tree solution is established.}

\DIFaddbegin \begin{definition}\label{def:ri-implementation}
\DIFadd{An RI implementation of the feedback $u=-Kx$ is a pair of matrices $(Y,U)$ such that $Y$ is nonsingular and $K=-UY^{-1}$.  Online evaluation consists of solving $Y\eta=x$ and then computing $u=U\eta$.
}\end{definition}

\DIFaddend Assume in this subsection that a stabilizing solution $H$ of the reduced DARE \eqref{eq:sr-reduced-H-DARE-expanded} is known.  By Lemma~\ref{lem:full-reduced-dare-equivalence}, the corresponding matrix $P$ in \eqref{eq:sr-P-block-form} is the stabilizing solution of the original DARE \eqref{eq:dare}.  Equation~\eqref{eq:sr-feedback-from-H-state} shows that the state $x_k=(s_k,d_k)$ enters the optimal feedback only through the storage-prediction variable \(w_k=s_k+Rd_k\).
Use the nonsingular matrix $Z$ defined in \eqref{eq:Z-definition} as the coordinate matrix for $w_k$.
Define the denominator block
\begin{equation}\label{eq:sr-known-H-Y}
    Y=
    \begin{bmatrix}
        Z&-R\\
        0&I_m
    \end{bmatrix}.
\end{equation}
Thus writing $x_k=Y\eta_k$ with $\eta_k=(\zeta_k,\delta_k)$ means
\begin{equation}\label{eq:sr-known-H-variables}
    s_k=Z\zeta_k-R\delta_k,
    \qquad
    d_k=\delta_k.
\end{equation}
Then
\begin{equation}\label{eq:sr-known-H-w-representation}
   w_k= s_k+Rd_k=Z\zeta_k.
\end{equation}
The matrix $Y$ is invertible whenever $Z$ is invertible: given $(s_k,d_k)$, first set $\delta_k=d_k$, then solve \(Z\zeta_k=s_k+Rd_k\).
For the sparse implementation question, it remains to derive the input block $U$.
Define
\begin{equation}\label{eq:sr-known-H-F}
    F=-rD^{-1}(-S^\top+\beta C^\top H).
\end{equation}
From \eqref{eq:sr-feedback-from-H-state} and \eqref{eq:sr-known-H-F}, the feedback satisfies $u_k=F w_k$.  Since \eqref{eq:sr-known-H-w-representation} gives $w_k=Z\zeta_k$, and the delay variable $\delta_k$ does not enter $w_k$, the input block is
\begin{equation}\label{eq:sr-known-H-U}
    U=\begin{bmatrix}FZ&0\end{bmatrix}.
\end{equation}
By Lemma~\ref{lem:tree-incidence-rank-Z}, the denominator block $Y$ in \eqref{eq:sr-known-H-Y} is graph-adapted and nonsingular for the tree choice $Z$ in \eqref{eq:Z-definition}: apart from the single mean-storage column, the columns in $C$ are edge-local.  The feedback produced by \eqref{eq:sr-known-H-U} is the same gain as \eqref{eq:sr-feedback-from-H-state}, because \eqref{eq:sr-known-H-w-representation} gives \(U\eta_k=FZ\zeta_k=Fw_k=u_k\).
Thus changing from $w_k$ to $\zeta_k$ changes only the coordinates used to apply the gain, not the optimal feedback law.  \REV{The construction is summarized in the following proposition.}

\DIFaddbegin \begin{proposition}\label{prop:ri-factorization}
\DIFadd{Under Assumption~\ref{ass:data}, let $H$ be the stabilizing solution of the reduced DARE \eqref{eq:sr-reduced-H-DARE-expanded}.  The matrices $Y$ and $U$ in \eqref{eq:sr-known-H-Y} and \eqref{eq:sr-known-H-U} form an RI implementation in the sense of Definition~\ref{def:ri-implementation}.
}\end{proposition}

\DIFaddend \subsection{Four-storage examples}\label{subsec:four-storage-sr-H}

The following four-storage sender--receiver example shows that the compressed matrix $FZ$ can be sparse even when $H$ is dense.  Using the standard basis vectors $\hat e_i$ in $\R^4$, take
\DIFaddbegin
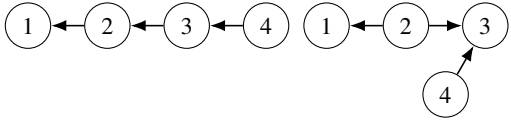
\begin{figure}[t]
\centering
\begin{tikzpicture}[x=1.05cm,y=1.05cm,vertex/.style={circle,draw,fill=white,minimum size=6mm,inner sep=0pt,font=\small},edge/.style={-{Latex[length=2mm]},semithick}]
  \node[font=\small] at (1.8,1.15) {(a) Outward path};
  \node[vertex] (p1) at (0,0) {1};
  \node[vertex] (p2) at (1,0) {2};
  \node[vertex] (p3) at (2,0) {3};
  \node[vertex] (p4) at (3,0) {4};
  \draw[edge] (p2)--(p1);
  \draw[edge] (p3)--(p2);
  \draw[edge] (p4)--(p3);
  \begin{scope}[xshift=5.0cm]
    \node[font=\small] at (0,1.15) {(b) Merging orientation};
    \node[vertex] (m2) at (0,0) {2};
    \node[vertex] (m1) at (-1,0) {1};
    \node[vertex] (m3) at (1,0) {3};
    \node[vertex] (m4) at (0.5,-0.866) {4};
    \draw[edge] (m2)--(m1);
    \draw[edge] (m2)--(m3);
    \draw[edge] (m4)--(m3);
  \end{scope}
\end{tikzpicture}
\caption{The two four-storage examples.  Arrows point from senders to receivers.  The outward path satisfies Assumption~\ref{ass:outward-rooted}; the merging orientation satisfies Assumption~\ref{ass:data} but not Assumption~\ref{ass:outward-rooted}.}
\label{fig:four-storage-graphs}
\end{figure}
\DIFaddend
\REV{The outward path is shown in Fig.~\ref{fig:four-storage-graphs}(a).}
\begin{equation}\label{eq:four-storage-sr-data}
    R=[\hat e_1\ \hat e_2\ \hat e_3],
    \qquad
    S=[\hat e_2\ \hat e_3\ \hat e_4].
\end{equation}
Since $C=R-S$, the data in \eqref{eq:four-storage-sr-data} describe a graph in which each edge sends material one step upward and receives it at the adjacent upstream vertex.  For the representative value $\beta=1/2$, direct substitution into \eqref{eq:sr-reduced-H-DARE-expanded} gives the solution
\begin{equation}\label{eq:four-storage-sr-H}
    H=\frac{1}{105}
    \begin{bmatrix}
        169&64&29&14\\
        64&64&29&14\\
        29&29&29&14\\
        14&14&14&14
    \end{bmatrix}.
\end{equation}
The matrix in \eqref{eq:four-storage-sr-H} is dense, but it has a nested form: its $(i,j)$ entry depends only on the downstream index $\max\{i,j\}$.  The edge differences in $C^\top H C$ therefore cancel almost completely, and the denominator appearing in the compressed feedback is diagonal in this example, even though $H$ itself is dense.

Substituting $\beta=1/2$ and $r=1/\sqrt2$ into the matrix $F$ from \eqref{eq:sr-known-H-F} gives
\begin{equation}\label{eq:four-storage-sr-FZ}
    FZ=\frac{1}{\sqrt2}
    \begin{bmatrix}
        -1&2/3&0&1/3\\
        0&-1&6/7&4/7\\
        0&0&-1&11/15
    \end{bmatrix}.
\end{equation}
In this sender--receiver example, the known-$H$ construction has a sparse compressed input block.  Equation \eqref{eq:four-storage-sr-FZ} suggests a possible mechanism for the general outward-tree case: $H$ may be dense in vertex variables but still have a hierarchical form whose edge differences make $D$ and $FZ$ sparse.  The task is therefore not to require $H$ itself to be sparse, but to characterize when its incidence compressions $C^\top H C$ and $C^\top HZ$ are sparse or recursively computable on the tree. 
These properties depend on both the graph and the Riccati solution: the graph restricts the possible nonzero entries, while the control problem determines their values.

We also record a nearby example that falls outside the outward-rooted assumptions.  Keep a four-vertex tree but orient the edges as $2\to1$, $2\to3$, and $4\to3$.  With the same standard basis vectors as in \eqref{eq:four-storage-sr-data}, take \(R=[\hat e_1\ \hat e_3\ \hat e_3], \qquad S=[\hat e_2\ \hat e_2\ \hat e_4]\).
This example satisfies Assumption~\ref{ass:data}, but $\rank R=2<3=m$, so it violates Assumption~\ref{ass:outward-rooted}.  For $\beta=1/2$, direct substitution into \eqref{eq:sr-reduced-H-DARE-expanded} gives the solution
\[
    H=
    \begin{bmatrix}
        4/3&1/3&1/3&1/3\\
        1/3&1/3&1/3&1/3\\
        1/3&1/3&4/3&1/3\\
        1/3&1/3&1/3&1/3
    \end{bmatrix}.
\]
For this solution, substituting $\beta=1/2$ and $r=1/\sqrt2$ into the matrix $F$ from \eqref{eq:sr-known-H-F} gives
\begin{equation}\label{eq:rank-deficient-sr-FZ}
    FZ=\frac{1}{\sqrt2}
    \begin{bmatrix}
        -1&0&0&1/3\\
        0&-1&0&0\\
        0&0&-1&1/3
    \end{bmatrix}.
\end{equation}
Equation~\eqref{eq:rank-deficient-sr-FZ} shows that this example has sparse compressed input factors although it is outside Assumption~\ref{ass:outward-rooted}.

\subsection{\DIFdelbegin \DIFdel{Intersection components}\DIFdelend \DIFaddbegin \DIFadd{Explicit reduced solution and RI support}\DIFaddend }\label{subsec:paths-outward-trees}

The \DIFdelbegin \DIFdel{rank-deficient example at the end of the previous subsection separates the tree condition in Assumption~\ref{ass:data} from the receiver-rank condition in Assumption~\ref{ass:outward-rooted}.  The RI denominator construction only needs an invertible graph-coordinate matrix $Z$, but }\DIFdelend \DIFaddbegin \DIFadd{two preceding examples indicate that incidence differences, rather than entrywise sparsity of $H$, determine the }\DIFaddend sparsity of the \DIFdelbegin \DIFdel{online input block is a separate statement.  From \(D=S^\top S+\beta C^\top H C, \qquad FZ=rD^{-1}(S^\top-\beta C^\top H)Z\), there are two distinct mechanisms to control: the incidence compression $C^\top H C$ must be sparse enough that $D$ is sparse or locally block sparse, and the product of $D^{-1}$ with $(S^\top-\beta C^\top H)Z$ must not destroy this locality.  These two requirements involve the compressed denominator and input matrices, not the rank of $R$ alone.  For the rank-deficient treeexample, \eqref{eq:rank-deficient-sr-FZ} verifies the input sparsity directly.  }
\DIFdel{For a directed path the compressed matrix $FZ$ can be written explicitly.  Let the vertices be $1,\ldots,N$ and the edges be $i+1\to i$, $i=1,\ldots,N-1$.  Let }\DIFdelend \DIFaddbegin \DIFadd{feedback factors.  We now derive $H$, $HC$, and $C^\top HC$ explicitly for every directed tree.  For each edge $e_i$, write $e_i:s_i\to t_i$ and let }\DIFaddend $c_i$ \DIFdelbegin \DIFdel{, $R_i$, and $S_i$ denote the }\DIFdelend \DIFaddbegin \DIFadd{be column }\DIFaddend $i$ \DIFdelbegin \DIFdel{th columns }\DIFdelend of $C$\DIFdelbegin \DIFdel{, $R$, and $S$.  Then \(c_i=\hat e_i-\hat e_{i+1}, \qquad R_i=\hat e_i, \qquad S_i=\hat e_{i+1}\).
Define positive scalars \(g_1=1, \qquad g_{i+1}=\frac{\beta g_i}{1+\beta g_i}, \qquad i=1,\ldots,N-2\), and
}
\DIFdel{\begin{gathered}
    h_N=\frac{\beta g_{N-1}}{(1-\beta)(1+\beta g_{N-1})},\\
    h_i=h_{i+1}+g_i,\qquad i=N-1,\ldots,1.
\end{gathered}
}
\DIFdel{Then the nested matrix $H_{ij}=h_{\max\{i,j\}}$ solves \eqref{eq:sr-reduced-H-DARE-expanded}.  In this special case
}
\DIFdel{\begin{gathered}
    C^\top H C=\operatorname{diag}(g_1,\ldots,g_{N-1}),\\
    D=\operatorname{diag}(1+\beta g_1,\ldots,1+\beta g_{N-1}).
\end{gathered}
}
\DIFdelend \DIFaddbegin \DIFadd{.
}

\begin{lemma}\label{lem:tree-height}
\DIFadd{Under Assumption~\ref{ass:data}, there is a function $h:V\to\mathbb Z_{\ge0}$ such that
}\begin{equation}\DIFadd{\label{eq:tree-height}
    h(t_i)=h(s_i)+1,
    \qquad e_i:s_i\to t_i,
}\end{equation}\DIFaddend 
and \DIFdelbegin \DIFdel{the nonzero entries of $FZ$ are }
    \DIFdel{(FZ)_{i,i}=-r, \qquad (FZ)_{i,i+1}=\frac{r}{1+\beta g_i} \quad (i<N-1),
}
\DIFdelend \DIFaddbegin \DIFadd{$\min_{v\in V}h(v)=0$.  This function is unique.
}\end{lemma}

\DIFadd{For $a\in\mathbb Z$, let
}\begin{equation}\DIFadd{\label{eq:height-superlevel-components}
    V_a=\{v\in V:h(v)\ge a\},
    \qquad
    }{\DIFadd{\mathcal C}}\DIFadd{_a=\operatorname{cc}(G}[\DIFadd{V_a}]\DIFadd{),
}\end{equation}\DIFaddend 
\DIFdelbegin \DIFdel{together with the mean-column entries
}
    \DIFdel{(FZ)_{i,N} =r\,\frac{1-\beta i g_i}{1+\beta g_i}.
}
\DIFdelend \DIFaddbegin \DIFadd{where $G[V_a]$ is the subgraph induced by $V_a$ and $\operatorname{cc}$ denotes its collection of connected components after edge orientations are ignored.  Let $h_{\max}=\max_{v\in V}h(v)$.  For $a\le h_{\max}$ and $\Omega\in{\mathcal C}_a$, define
}\begin{equation}\DIFadd{\label{eq:height-component-kappa}
    \kappa_{a,\Omega}
    =\left(\sum_{v\in\Omega}\beta^{a-h(v)}\right)^{-1}.
}\end{equation}\DIFaddend 
\DIFdelbegin \DIFdel{Thus the directed path gives a positive sparse answer to the known-$H$ compressed-feedback question by direct calculation.  }\DIFdelend \DIFaddbegin \DIFadd{We call $\Omega\in{\mathcal C}_a$ a height component at level $a$.
For every subset $\Omega\subseteq V$, let $\xi_\Omega$ denote its indicator vector.
}\DIFaddend 

\DIFaddbegin
\REV{For the path in Fig.~\ref{fig:four-storage-graphs}(a), oriented $4\to3\to2\to1$, the minimum-zero height function is}
\[
    h(4)=0,\qquad h(3)=1,\qquad h(2)=2,\qquad h(1)=3.
\]
The positive-level height components and their coefficients are
\begin{align*}
    {\mathcal C}_1&=\{\{1,2,3\}\},
    &\kappa_{1,\{1,2,3\}}&=(1+\beta^{-1}+\beta^{-2})^{-1},\\
    {\mathcal C}_2&=\{\{1,2\}\},
    &\kappa_{2,\{1,2\}}&=(1+\beta^{-1})^{-1},\\
    {\mathcal C}_3&=\{\{1\}\},
    &\kappa_{3,\{1\}}&=1.
\end{align*}
Thus the components become nested as the level increases.  For a general orientation they need not be rooted subtrees.
\DIFaddend

\DIFaddbegin
\DIFadd{The role of the height components can be understood from flow cancellation.  Each directed edge joins consecutive height levels by \eqref{eq:tree-height}.  When the storage-balance equations are summed over a height component, every transport variable on an internal edge appears once with each sign and cancels.  The remaining quantity is the aggregate storage coordinate $\xi_\Omega^\top w$.  Minimizing the discounted quadratic cost for a fixed aggregate produces the coefficient $\kappa_{a,\Omega}$ in \eqref{eq:height-component-kappa} and a rank-one term proportional to $\xi_\Omega\xi_\Omega^\top$.  The following result shows that these component contributions give the complete stabilizing solution.}
\DIFaddend

\DIFdelbegin 
\DIFdelend \DIFaddbegin \begin{theorem}\label{thm:explicit-height-solution}
\DIFaddend Under Assumption~\ref{ass:data}, \DIFdelbegin \DIFdel{suppose }\DIFdelend \DIFaddbegin \DIFadd{the stabilizing positive semidefinite solution of the reduced DARE \eqref{eq:sr-reduced-H-DARE-expanded} is
}\begin{equation}\DIFadd{\label{eq:explicit-height-solution}
\begin{aligned}
    H={}&q_0\ones\ones^\top
    +\sum_{a=1}^{h_{\max}}
      \sum_{\Omega\in{\mathcal C}_a}
      \kappa_{a,\Omega}\xi_\Omega\xi_\Omega^\top,\\
    q_0={}&
    \frac{1}{(1-\beta)\sum_{v\in V}\beta^{-h(v)}}.
\end{aligned}
}\end{equation}
\DIFadd{\REV{Every coefficient in \eqref{eq:explicit-height-solution} is positive.}}\end{theorem}

\DIFadd{For edge $e_i:s_i\to t_i$, set
}\DIFaddend \begin{equation}\DIFdelbegin 
\DIFdel{c_i^\top Hc_j
    =\sum_{\Omega\in{\mathcal L}:\, r_j\in \Omega,\ s_j\notin \Omega}
      q_\Omega}
\DIFdel{1_{r_i\in \Omega}-1_{s_i\in \Omega}}
\DIFdel{.
}\DIFdelend \DIFaddbegin \label{eq:edge-height-component}
\begin{aligned}
    a_i&=h(t_i),\\
    \Omega_i&=\text{the component of $G[V_{a_i}]$ containing $t_i$},\\
    \kappa_i&=\kappa_{a_i,\Omega_i}.
\end{aligned}
\DIFaddend \end{equation}
\DIFdelbegin 
\DIFdelend 

\DIFdelbegin 
\DIFdelend \DIFaddbegin \begin{corollary}\label{cor:exact-incidence-compression}
\DIFaddend Under Assumption~\ref{ass:data}, \DIFdelbegin \DIFdel{for every $z\in{\mathcal W}_{\mathcal B}$ there is a vector $y\in\R^m$ such that
}\DIFdel{
    Sy=z,
    \qquad
    Cy\in{\mathcal W}_{\mathcal B}.
}
\DIFdelend \DIFaddbegin \DIFadd{the following identities hold:
}\begin{align}
    \DIFadd{Hc_i}&\DIFadd{=\kappa_i\xi_{\Omega_i},
        \label{eq:exact-HC-column}}\\
    \DIFadd{c_j^\top Hc_i}&\DIFadd{=
    \begin{cases}
        \kappa_i,&(a_j,\Omega_j)=(a_i,\Omega_i),\\
        0,&\text{otherwise}.
    \end{cases}
        \label{eq:exact-CtHC-entry}
}\end{align}\DIFaddend 
\DIFdelbegin 

\DIFdel{The rank-deficient tree example associated with \eqref{eq:rank-deficient-sr-FZ} is a simple case where ${\mathcal L}=\{V,\{1\},\{3\}\}$ and \(H=\frac{1}{3}\ones\ones^\top+ \xi_{\{1\}}\xi_{\{1\}}^\top+ \xi_{\{3\}}\xi_{\{3\}}^\top\).
This gives
}
    \DIFdel{HC=
    \begin{bmatrix}
        1&0&0\\
        0&0&0\\
        0&1&1\\
        0&0&0
    \end{bmatrix},
    \qquad
    C^\top HC=
    \begin{bmatrix}
        1&0&0\\
        0&1&1\\
        0&1&1
    \end{bmatrix},
}
\DIFdel{which gives the sparsity of the compressed input block through \eqref{eq:rank-deficient-sr-FZ}.  The example with edges $2\to1$, $1\to3$, $4\to3$, and $5\to4$ shows why intersections are needed:
there
}
    \DIFdel{H=\frac{1}{5}\ones\ones^\top+
      \frac{1}{4}\xi_{\{1,3,4\}}\xi_{\{1,3,4\}}^\top+
      \xi_{\{3\}}\xi_{\{3\}}^\top,
}
\DIFdel{where $\{1,3,4\}={\mathcal R}(1)\cap{\mathcal R}(4)$.  The corresponding compression is sparse, }
    \DIFdel{C^\top HC=
    \begin{bmatrix}
        1/4&0&0&1/4\\
        0&1&1&0\\
        0&1&1&0\\
        1/4&0&0&1/4
    \end{bmatrix}.
}

\DIFdel{These examples motivate considering the set of matrices 
${\mathcal V}_{\mathcal L}$ in \eqref{eq:intersection-component-space}
as a refinement capturing observed structural patterns.  }

\DIFdel{Let ${\mathcal A}$ be a collection of subsets of $V$.  Suppose
}\DIFdel{
    M=\sum_{\Omega\in{\mathcal A}}q_\Omega\xi_\Omega\xi_\Omega^\top .
}
\DIFdelend \DIFdelbegin \DIFdel{If $e_i:s_i\to r_i$ is a directed edge and $c_i$ is column $i$ of $C$, then
}\DIFdel{
    Mc_i=\sum_{\Omega\in{\mathcal A}}q_\Omega
    \bigl(1_{r_i\in\Omega}-1_{s_i\in\Omega}\bigr)\xi_\Omega .
}

\DIFdel{Under Assumption~\ref{ass:data}, let ${\mathcal B}$ be the collection of subsets used in \eqref{eq:boolean-atom-space}.
If $H\in{\mathcal V}_{\mathcal B}$ and
}\DIFdel{
    D(H)=S^\top S+\beta C^\top HC
}
\DIFdel{is nonsingular,
then
}
\DIFdelend \DIFaddbegin \DIFadd{Define $\Gamma_{a,\Omega}=\{i:(a_i,\Omega_i)=(a,\Omega)\}$.  For each nonempty $\Gamma_{a,\Omega}$, let $\ones_{\Gamma_{a,\Omega}}$ denote its all-ones vector.  Let $\operatorname{blkdiag}$ denote a block-diagonal matrix with the listed diagonal blocks.  After grouping the edge indices by the nonempty sets $\Gamma_{a,\Omega}$,
}\DIFaddend \begin{equation}\DIFdelbegin 
\DIFdelend \DIFaddbegin \label{eq:exact-CtHC-blocks}
    \DIFadd{C^\top HC
    =}\operatorname{blkdiag}\DIFadd{_{(a,\Omega)}
    }\left(
      \DIFadd{\kappa_{a,\Omega}
      \ones_{\Gamma_{a,\Omega}}
      \ones_{\Gamma_{a,\Omega}}^\top
    }\right)\DIFadd{.
}\DIFaddend \end{equation}
\DIFdelbegin 
\DIFdelend \DIFaddbegin \end{corollary}
\DIFaddend 

\DIFdelbegin 
\DIFdel{Under Assumption~\ref{ass:data}, let $H$ be the stabilizing positive semidefinite solution of }\DIFdelend \DIFaddbegin \DIFadd{Theorem~\ref{thm:explicit-height-solution} removes the need to solve }\DIFaddend the reduced DARE \DIFdelbegin \DIFdel{\eqref{eq:sr-reduced-H-DARE-expanded}.  Then $H\in{\mathcal V}_{\mathcal B}$.
}
\DIFdelend \DIFaddbegin \DIFadd{numerically for the stage cost \REV{in \eqref{eq:control-problem}}.  The directed tree and $\beta$ determine $h$, the height components ${\mathcal C}_a$, and all coefficients in \eqref{eq:explicit-height-solution}.  Corollary~\ref{cor:exact-incidence-compression} then gives $C^\top HC$ in \eqref{eq:sr-DH} directly.
}\DIFaddend 

\REV{We now use the incidence compression in Corollary~\ref{cor:exact-incidence-compression} to bound the possible nonzero entries of the RI input factor.}

\DIFdelbegin 
\DIFdel{Proposition~\ref{prop:intersection-span-invariance} cannot be extended to ${\mathcal V}_{\mathcal L}$.  A counterexample is the tree $5\to1\to2\to3\leftarrow4\leftarrow6$.
}

\DIFdel{We have performed an exhaustive search over all directed trees with up to and including six vertices.  In all tested cases
the stabilizing positive semidefinite solution from Lemma~\ref{lem:reduced-dare-existence-conditions} belonged to ${\mathcal V}_{\mathcal L}$, suggesting that this structure may hold more generally as a refinement of the provable class $\mathcal{V}_\mathcal B$.  }

\DIFdel{Under Assumption~\ref{ass:data}, let $H$ be the stabilizing positive semidefinite solution from Lemma~\ref{lem:reduced-dare-existence-conditions}.
Then $H\in{\mathcal V}_{\mathcal L}$.
}

\DIFdel{Use the edge enumeration $E=\{e_1,\ldots,e_m\}$.  }\DIFdelend Define binary matrices $P_D,P_G\in\{0,1\}^{m\times m}$ \DIFdelbegin \DIFdel{so that row $p$ and column $q$ correspond to edges $e_p$ and $e_q$, with entries
}
\DIFdel{\begin{aligned}
    P_D(p,q)=1
    \quad\Longleftrightarrow\quad
    &\text{same sender, or}\\
    &\text{common boundary crossing in }{\mathcal L},
\end{aligned}
}
\DIFdelend \DIFaddbegin \DIFadd{by
}\begin{align*}
\DIFadd{P_D(p,q)=1
}&\DIFadd{\quad\Longleftrightarrow\quad
s_p=s_q\ \text{or}\ (a_p,\Omega_p)=(a_q,\Omega_q),}\\
\DIFadd{P_G(p,q)=1
}&\DIFadd{\quad\Longleftrightarrow\quad
C_{s_p,q}\ne0\ \text{or}\ (a_p,\Omega_p)=(a_q,\Omega_q).
}\end{align*}\DIFaddend 
\DIFdelbegin \DIFdel{and
}
\DIFdel{\begin{aligned}
    P_G(p,q)=1
    \quad\Longleftrightarrow\quad
    &C_{s(e_p),q}\ne0,\text{ or}\\
    &\text{common boundary crossing in }{\mathcal L}.
\end{aligned}
}
\DIFdel{Set $P_D(p,q)=0$ and $P_G(p,q)=0$ when the corresponding }\DIFdelend \DIFaddbegin \DIFadd{Set an entry to zero when its displayed }\DIFaddend condition fails.  \DIFdelbegin \DIFdel{Let $P_D^*$ be }\DIFdelend \DIFaddbegin \DIFadd{For any binary matrix $\Pi$, define $\Pi^*(p,q)=1$ if there is a chain $i_0=p,i_1,\ldots,i_\ell=q$ with $\Pi(i_{j-1},i_j)=1$ for every $j$, and set $\Pi^*(p,q)=0$ otherwise.  We call $\Pi^*$ }\DIFaddend the transitive closure of \DIFdelbegin \DIFdel{$P_D$.  For any matrix $M\in\R^{m\times N}$, define the }\DIFdelend \DIFaddbegin \DIFadd{$\Pi$.  For a matrix in $\R^{m\times N}$, its }\DIFaddend edge-column support \DIFdelbegin \DIFdel{of $MZ$ to be the binary matrix $P\in\{0,1\}^{m\times m}$ with $P(p,q)=1$ exactly when $(MZ)_{pq}\ne0$, for $q=1,\ldots,m$}\DIFdelend \DIFaddbegin \DIFadd{is the binary support of its first $m$ columns, which correspond to the incidence columns of $Z$ in \eqref{eq:Z-definition}}\DIFaddend .  The last column \DIFaddbegin \DIFadd{corresponds to the $\ones$ column }\DIFaddend of \DIFdelbegin \DIFdel{$Z$ from \eqref{eq:Z-definition} is the mean column .
}\DIFdelend \DIFaddbegin \DIFadd{$Z$.  In Boolean matrix arithmetic, multiplication replaces scalar multiplication by logical ``and'' and scalar addition by logical ``or''; inequalities are entrywise.
}\DIFaddend 

\begin{lemma}\label{lem:transitive-closure-classes}
Let \DIFdelbegin \DIFdel{$P\in\{0,1\}^{m\times m}$ }\DIFdelend \DIFaddbegin \DIFadd{$\Pi\in\{0,1\}^{m\times m}$ }\DIFaddend be symmetric and satisfy \DIFdelbegin \DIFdel{$P(p,p)=1$ }\DIFdelend \DIFaddbegin \DIFadd{$\Pi(p,p)=1$ }\DIFaddend for every $p$.  Let \DIFdelbegin \DIFdel{$P^*$ be the transitive closureof $P$.  Then $P^*(p,q)=1$ exactly when a chain $i_0=p,i_1,\ldots,i_\ell=q$ has $P(i_{h-1},i_h)=1$ for every $h$}\DIFdelend \DIFaddbegin \DIFadd{$\Pi^*$ be its transitive closure}\DIFaddend .  The relation $p\sim q$ defined by \DIFdelbegin \DIFdel{$P^*(p,q)=1$ }\DIFdelend \DIFaddbegin \DIFadd{$\Pi^*(p,q)=1$ }\DIFaddend is an equivalence relation.  If \DIFdelbegin \DIFdel{a matrix $A\in\R^{m\times m}$ satisfies $A_{pq}=0$ whenever $P(p,q)=0$, then $A$ }\DIFdelend \DIFaddbegin \DIFadd{$X_{pq}=0$ whenever $\Pi(p,q)=0$, then $X$ }\DIFaddend is block diagonal over the equivalence classes of $\sim$.
\end{lemma}

\begin{proposition}\label{prop:general-separator-FZ-sparsity}
\DIFdelbegin \DIFdel{In addition to }\DIFdelend \DIFaddbegin \DIFadd{Under }\DIFaddend Assumption~\ref{ass:data}, \DIFdelbegin \DIFdel{assume that the stabilizing positive semidefinite solution $H$ of the reduced DARE \eqref{eq:sr-reduced-H-DARE-expanded} belongs to ${\mathcal V}_{\mathcal L}$.  Then the }\DIFdelend \DIFaddbegin \DIFadd{the }\DIFaddend edge-column support $P_F$ of \DIFdelbegin \DIFdel{the product }\DIFdelend $FZ$, with $F$ from \eqref{eq:sr-known-H-F}\DIFdelbegin \DIFdel{and $Z$ from \eqref{eq:Z-definition}}\DIFdelend , satisfies
\begin{equation}\label{eq:separator-FZ-mask}
    P_F\le P_D^*P_G
\end{equation}
in Boolean matrix arithmetic.  No zero is guaranteed in the \DIFdelbegin \DIFdel{mean column }\DIFdelend \DIFaddbegin \DIFadd{column corresponding to $\ones$}\DIFaddend .
\end{proposition}

\DIFdelbegin 
\DIFdel{An analogous bound can be obtained when $H\in{\mathcal V}_{\mathcal B}$ by replacing common boundary crossing in ${\mathcal L}$ with common boundary crossing in ${\mathcal B}$ in the definitions of $P_D$ and $P_G$.  This bound is more conservative because the collection ${\mathcal B}$ can add boundary crossings and hence can add nonzero entries to the masks.
}

\DIFdel{Computable support count.}
\DIFdel{Proposition~\ref{prop:general-separator-FZ-sparsity} is entirely combinatorial once $R$ and $S$ have been specified: these matrices determine $C=R-S$, the receiver-side vertex-set family ${\mathcal L}$, and the masks $P_D$ and $P_G$.  The number of forced zero edge-columns in $FZ$ is at least \(m^2-\lvert\{(p,q):(P_D^*P_G)(p,q)=1\}\rvert\).
There is no additional guaranteed zero in the mean column.  Thus 
a conservative zero-count bound can be computed before solving the Riccati equation.  Equivalently, since }\DIFdelend \DIFaddbegin \DIFadd{Let $\operatorname{nnz}(X)$ denote the number of nonzero entries of a matrix $X$.  Since }\DIFaddend $U=[FZ\;0]$\DIFdelbegin \DIFdel{has zero delay columns, the graph-only upper bound
}\DIFdelend \DIFaddbegin \DIFadd{, Proposition~\ref{prop:general-separator-FZ-sparsity} gives
}\DIFaddend \begin{equation}\label{eq:graph-mask-U-bound}
    \operatorname{nnz}(U)
    \le
    \lvert\{(p,q):(P_D^*P_G)(p,q)=1\}\rvert+m\DIFaddbegin \DIFadd{.
}\DIFaddend \end{equation}
\DIFdelbegin \DIFdel{can be evaluated from the edge orientations before solving for $H$.  }\DIFdelend The first term counts \DIFdelbegin \DIFdel{the }\DIFdelend possible nonzeros in the \DIFdelbegin \DIFdel{first $m$ }\DIFdelend \DIFaddbegin \DIFadd{incidence }\DIFaddend columns of $FZ$, \DIFdelbegin \DIFdel{while }\DIFdelend \DIFaddbegin \DIFadd{and }\DIFaddend the final $m$ \DIFdelbegin \DIFdel{is the safe contribution of the mean column }\DIFdelend \DIFaddbegin \DIFadd{allows all entries in the column corresponding to $\ones$ to be nonzero}\DIFaddend .

\DIFdelbegin \DIFdel{Under Assumption~\ref{ass:outward-rooted}, Lemma~\ref{lem:outward-rooted-orientation} gives the root vertex $r$.  For each non-root vertex $v$, deleting the edge $(p(v),v)$ leaves $T_v$ as the connected component containing $v$, so ${\mathcal R}(v)=T_v$.  Nonempty intersections of such sets are again rooted subtrees: if $\cap_{v\in J}T_v\ne\emptyset$, then the vertices in $J$ lie on one root-to-leaf path and the intersection equals $T_w$, where $w$ is the vertex in $J$ farthest from $r$.  Hence, in the outward-rooted case, \eqref{eq:component-lattice} gives
}\DIFdel{
    }{\DIFdel{\mathcal L}}\DIFdel{=\{T_v:v\in V\}.
}

\DIFdelend \begin{corollary}\DIFdelbegin 
\DIFdelend \DIFaddbegin \label{cor:outward-rooted-FZ-support}
\DIFaddend Under Assumption~\ref{ass:outward-rooted}, let \DIFdelbegin \DIFdel{$H$ be the stabilizing positive semidefinite solution }\DIFdelend \DIFaddbegin \DIFadd{$v_{\rm root}$ be the root }\DIFaddend from Lemma~\DIFdelbegin \DIFdel{\ref{lem:reduced-dare-existence-conditions}.  Assume $H\in{\mathcal V}_{\mathcal L}$.  Then }\DIFdel{
    H\in\operatorname{span}\{\chi_v\chi_v^\top:v\in V\}.
}
\DIFdelend \DIFaddbegin \DIFadd{\ref{lem:outward-rooted-orientation}.  Then $h(v)$ equals the number of edges on the path from $v_{\rm root}$ to $v$, and
}\begin{equation}\DIFadd{\label{eq:outward-subtree-span}
    H=q_0\ones\ones^\top
      +\sum_{v\ne v_{\rm root}}
      \left(\sum_{x\in T_v}\beta^{h(v)-h(x)}\right)^{-1}
      \chi_v\chi_v^\top.
}\end{equation}\DIFaddend 
\DIFdelbegin 

\DIFdel{Under Assumption~\ref{ass:outward-rooted}, let $H$ be the stabilizing positive semidefinite solution from Lemma~\ref{lem:reduced-dare-existence-conditions}, and
assume $H\in{\mathcal V}_{\mathcal L}$.  Let $F$ be the matrix in \eqref{eq:sr-known-H-F}, and let $Z$ be the matrix in \eqref{eq:Z-definition}}\DIFdelend \DIFaddbegin \DIFadd{Moreover, $C^\top HC$ is diagonal}\DIFaddend .  Let $I_v$ be the set of indices \DIFdelbegin \DIFdel{$i$ such that edge $e_i$ leaves vertex }\DIFdelend \DIFaddbegin \DIFadd{of edges leaving }\DIFaddend $v$, and let $i^-(v)$ be the index of the \DIFdelbegin \DIFdel{unique edge entering }\DIFdelend \DIFaddbegin \DIFadd{edge entering a non-root vertex }\DIFaddend $v$\DIFdelbegin \DIFdel{, when $v$ is not the root.  If edge $e_i$ has sender $p_i$, then the }\DIFdelend \DIFaddbegin \DIFadd{.  The }\DIFaddend nonzero entries in row $i$ of the first $m$ columns of $FZ$ are contained in
\DIFdelbegin \DIFdel{the columns
}\DIFdelend \begin{equation}\label{eq:outward-row-mask-U-bound}
    I\DIFdelbegin \DIFdel{_{p_i}}\DIFdelend \DIFaddbegin \DIFadd{_{s_i}}\DIFaddend \cup\{i^-(\DIFdelbegin \DIFdel{p}\DIFdelend \DIFaddbegin \DIFadd{s}\DIFaddend _i):\DIFdelbegin \DIFdel{\ p}\DIFdelend \DIFaddbegin \DIFadd{s}\DIFaddend _i\ne \DIFdelbegin \DIFdel{r}\DIFdelend \DIFaddbegin \DIFadd{v_{\rm root}}\DIFaddend \}.
\end{equation}
The \DIFdelbegin \DIFdel{mean column }\DIFdelend \DIFaddbegin \DIFadd{column corresponding to $\ones$ }\DIFaddend may be nonzero.
\end{corollary}

\DIFdelbegin \DIFdel{The row-wise bound in }\DIFdelend Corollary~\ref{cor:outward-rooted-FZ-support} \DIFdelbegin \DIFdel{is conditional on $H\in{\mathcal V}_{\mathcal L}$.  Under this hypothesis, \eqref{eq:sr-known-H-U} gives $U=[FZ\;0]$, so row $i$ of $U$ has support contained in the columns in \eqref{eq:outward-row-mask-U-bound}, plus the mean column.  Hence \(\operatorname{nnz}(U) \le \sum_{i=1}^m\bigl(|I_{p_i}|+1_{\{p_i\ne r\}}+1\bigr)\),
}\DIFdelend \DIFaddbegin \DIFadd{and \eqref{eq:sr-known-H-U} give
}\[
\DIFadd{\operatorname{nnz}(U)
\le
\sum_{i=1}^m
\left(|I_{s_i}|+1_{\{s_i\ne v_{\rm root}\}}+1\right),
}\]
\DIFaddend where the last \DIFdelbegin \DIFdel{$1$ accounts for the mean column }\DIFdelend \DIFaddbegin \DIFadd{term counts the column corresponding to $\ones$.  For a directed path, each set $I_v$ has at most one element, so the number of nonzeros in $U$ grows linearly with $N$}\DIFaddend .
\DIFdelbegin \DIFdel{If one separator component is large, $D^{-1}$ may be dense on that component; Assumption~\ref{ass:outward-rooted} is useful precisely because it makes the components small and identifiable as sibling and parent-sibling neighborhoods.
}

\DIFdel{For example, the valid sender--receiver branching tree \(1\to2, \qquad 1\to3, \qquad 2\to4, \qquad 2\to5\)
with $\beta=1/2$ has the reduced solution
}
    \DIFdel{H=
    \begin{bmatrix}
        0.1538&0.1538&0.1538&0.1538&0.1538\\
        0.1538&0.3538&0.1538&0.3538&0.3538\\
        0.1538&0.1538&1.1538&0.1538&0.1538\\
        0.1538&0.3538&0.1538&1.3538&0.3538\\
        0.1538&0.3538&0.1538&0.3538&1.3538
    \end{bmatrix},
}
\DIFdel{which lies in the subtree span in \eqref{eq:outward-subtree-span}: it is a sum of one global rank-one term, one subtree term for the branch rooted at vertex $2$, and leaf-subtree terms.  In this case
}
    \DIFdel{D=
    \begin{bmatrix}
        1.1&1&0&0\\
        1&1.5&0&0\\
        0&0&1.5&1\\
        0&0&1&1.5
    \end{bmatrix},
}
\DIFdel{which is block diagonal over the sibling groups leaving vertices $1$ and $2$, and
}
    \DIFdel{FZ=\frac{1}{\sqrt2}
    \begin{bmatrix}
        -1&0&0&0&0.8462\\
        0&-1&0&0&-0.2308\\
        0.4&0&-1&0&0.2\\
        0.4&0&0&-1&0.2
    \end{bmatrix}.
}
\DIFdel{Thus the same graph-adapted denominator basis works for this branching tree; only sibling-block solves replace the scalar denominators of the path case.
}\DIFdelend

\begingroup\color{black}
\subsection{Input penalty}\label{subsec:input-penalty}

An input penalty preserves the coordinate factor $Y$, but the explicit
Riccati formula and graph sparsity guarantees generally no longer apply.
We give sufficient conditions under which rows or columns of the input
factor retain their values and hence their zeros.

For $W\succeq0$, add $\sum_{k=0}^{\infty}u_k^\top Wu_k$ to the cost.
Let $H_W$ and $u_k=F_W(s_k+Rd_k)$ be the corresponding optimal reduced
Riccati matrix and stationary feedback. Write $F_0=F$, $H_0=H$, and
$D_0=D$ for the matrices in \eqref{eq:sr-known-H-F} and
\eqref{eq:sr-DH}, and set
\begin{equation}\label{eq:weighted-invariant-subspace}
T_0=rI+CF_0,\qquad
\mathcal V_W=\bigcap_{k=0}^{N-1}\ker(WF_0T_0^k).
\end{equation}
This is the largest $T_0$-invariant subspace in $\ker(WF_0)$.
The weighted reduced DARE uses
\begin{equation}\label{eq:weighted-denominator}
D_W=W+S^\top S+\beta C^\top H_W C
\end{equation}
in place of $D$ in \eqref{eq:sr-DH}. The coordinate factor $Y$ in
\eqref{eq:sr-known-H-Y} still applies. The explicit Riccati formula
\eqref{eq:explicit-height-solution}, incidence compression
\eqref{eq:exact-CtHC-entry}, and graph support bound
\eqref{eq:separator-FZ-mask} and count bound
\eqref{eq:graph-mask-U-bound} rely on $W=0$: the flow allocation in the
proof of Theorem~\ref{thm:explicit-height-solution} minimizes storage
cost alone and does not apply unchanged when $W\ne0$.

For the path in \eqref{eq:four-storage-sr-data} with $\beta=1/2$ and
$W=0.1I_3$, numerical solution of the weighted reduced DARE gives
$(F_WC)_{13}\approx0.009826$, whereas \eqref{eq:four-storage-sr-FZ}
gives $(F_0C)_{13}=0$. The input-factor count rises from
$\operatorname{nnz}([F_0Z\;0])=8$ to $\operatorname{nnz}([F_WZ\;0])=12$,
exceeding the zero-penalty graph bound $8$ from
\eqref{eq:graph-mask-U-bound}.

\begin{proposition}\label{prop:weighted-retained-zeros}
Under Assumption~\ref{ass:data}, let $z_j$ be column $j$ of $Z$ in
\eqref{eq:Z-definition}. If $z_j\in\mathcal V_W$, then
$F_Wz_j=F_0z_j$. For input index $i$, if
\begin{equation}\label{eq:weighted-row-test}
W D_0^{-1}\hat e_i=0,\qquad
C D_0^{-1}\hat e_i\in\mathcal V_W,
\end{equation}
then row $i$ of $F_WZ$ equals row $i$ of $F_0Z$.
\end{proposition}

\begin{remark}
For \eqref{eq:four-storage-sr-data} at $\beta=1/2$ with
$W=\operatorname{diag}(0,0,\eta)$, $\eta>0$, equations
\eqref{eq:sr-DH}, \eqref{eq:four-storage-sr-H},
\eqref{eq:four-storage-sr-FZ}, and \eqref{eq:weighted-invariant-subspace}
verify the hypotheses of Proposition~\ref{prop:weighted-retained-zeros}
for rows and incidence columns $1,2$ of $F_WZ$, which therefore equal
those of $F_0Z$.
Yet $P_W\ne P_0$, since $WF_0C\hat e_3=-\eta r\hat e_3\ne0$:
equality would require the unique unweighted optimum from
$(C\hat e_3,0)$ to incur zero input penalty.
\end{remark}
\endgroup

\section{Online feedback implementation}\label{sec:implementation-parallelism}

\REV{Under Assumption~\ref{ass:data}, Theorem~\ref{thm:explicit-height-solution} gives the stabilizing reduced solution explicitly.}  The online RI feedback is applied by a denominator solve followed by the input multiplication.  Proposition~\ref{prop:general-separator-FZ-sparsity} and \eqref{eq:graph-mask-U-bound} give the graph-mask bound for the final product.  Corollary~\ref{cor:outward-rooted-FZ-support} gives a sharper row support for outward-rooted trees.  Given $x_k=(s_k,d_k)$, first solve \eqref{eq:sr-known-H-Y} for the variables in \eqref{eq:sr-known-H-variables}.  Equivalently, set $\delta_k=d_k$ and solve \eqref{eq:sr-known-H-w-representation} for $(\eta^E_k,\eta^0_k)$.  Then apply \eqref{eq:sr-known-H-U}.

For clarity, the online RI controller can be applied as follows after the offline construction of $Y$ and $U=[FZ\;0]$.
\begin{enumerate}
    \item Read the current state $x_k=(s_k,d_k)$ and set the delay component $\delta_k=d_k$.
    \item Solve the denominator equation \eqref{eq:sr-known-H-Y}, or equivalently \eqref{eq:sr-known-H-w-representation}, for the graph variables $(\eta^E_k,\eta^0_k)$.
    \item Form the compressed coordinate vector $\eta_k$ from $(\eta^E_k,\eta^0_k,\delta_k)$.
    \item Evaluate the control by the sparse product $u_k=U\eta_k$.
\end{enumerate}
The offline stage therefore stores $Y$, its chosen sparse factorization or tree-solve data, and $U$; the online stage consists only of the denominator solve and the final sparse input multiplication.

The solve in \eqref{eq:sr-known-H-Y} is a tree computation.  \REV{Summing \eqref{eq:sr-known-H-w-representation} gives $\eta^0_k=\ones^\top(s_k+Rd_k)/N$. Root the undirected tree arbitrarily. Delete each edge and sum $s_k+Rd_k-\eta^0_k\ones$ over the component containing its child. Its coordinate in $\eta^E_k$ is this sum if directed toward the child, and its negative otherwise. Leaf-to-root messages compute these sums.}  Disjoint subtrees can be processed in parallel until their messages meet.

After the variables in \eqref{eq:sr-known-H-variables} are available, the final step is only the sparse matrix-vector product in \eqref{eq:sr-known-H-U}.  \DIFdelbegin \DIFdel{If }\DIFdelend \DIFaddbegin \DIFadd{Under }\DIFaddend Assumption~\ref{ass:data}\DIFdelbegin \DIFdel{holds and the stabilizing positive semidefinite solution $H$ of \eqref{eq:sr-reduced-H-DARE-expanded} belongs to ${\mathcal V}_{\mathcal L}$, the }\DIFdelend \DIFaddbegin \DIFadd{, Theorem~\ref{thm:explicit-height-solution} and the }\DIFaddend graph mask in \eqref{eq:separator-FZ-mask} \DIFdelbegin \DIFdel{gives }\DIFdelend \DIFaddbegin \DIFadd{give }\DIFaddend the stored edge-column supports.  Under Assumption~\ref{ass:outward-rooted}, Corollary~\ref{cor:outward-rooted-FZ-support} gives the smaller row supports in \eqref{eq:outward-row-mask-U-bound}.  Thus the online implementation consists of a solve using \eqref{eq:sr-known-H-Y} followed by a sparse matrix-vector product using \eqref{eq:sr-known-H-U}.  The sparse matrix-vector product is parallel over the nonzero rows and row supports of the stored matrix.

\section{Numerical results}\label{sec:numerical-li-ri-comparison}

We compare the factor sparsity of the RI implementation with the factor sparsity displayed by Adlercreutz and Pates~\cite{AdlercreutzPates2026}.  The comparison covers their path example, one larger directed tree, and three orientations of a five-vertex tree.  One orientation has two edges with the same receiving vertex and therefore violates Assumption~\ref{ass:outward-rooted}, although it satisfies Assumption~\ref{ass:data}.  On the LI side we report $K_1$ and $K_2$.  \REV{The LI online step multiplies $K_2x$ and then solves $K_1u=K_2x$, with $K_1\in\R^{m\times m}$.}  On the RI side we report $Y$ and $U=[FZ\;0]$.  \REV{By Definition~\ref{def:ri-implementation}, the RI step solves $Y\eta=x$ and then multiplies $u=U\eta$, with $Y\in\R^{(N+m)\times(N+m)}$.}  Numerical entries are displayed only for the path, because Adlercreutz and Pates report only sparsity patterns for the other examples. \REV{We use $\beta=1/2$ throughout.}

\DIFaddbegin
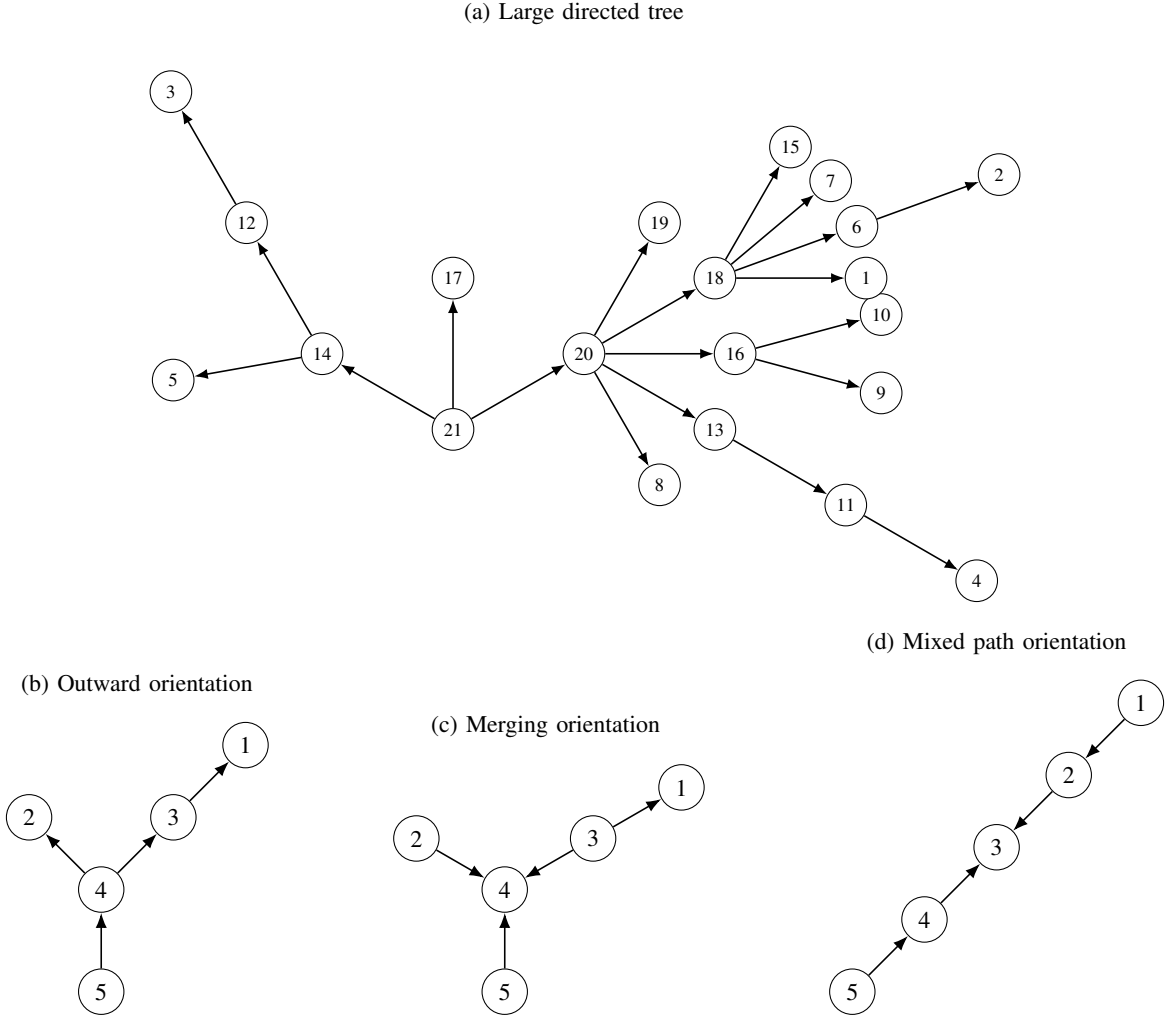
\begin{figure*}[t]
\centering
\begin{tikzpicture}[x=2cm,y=2cm,vertex/.style={circle,draw,fill=white,minimum size=5.5mm,inner sep=0pt,font=\scriptsize},edge/.style={-{Latex[length=1.8mm]},semithick}]
  \node[font=\small] at (0.8,2.75) {(a) Large directed tree};
  \node[vertex] (v21) at (0,0) {21};
  \node[vertex] (v14) at (-0.866,0.5) {14};
  \node[vertex] (v17) at (0,1) {17};
  \node[vertex] (v20) at (0.866,0.5) {20};
  \node[vertex] (v5) at (-1.851,0.326) {5};
  \node[vertex] (v12) at (-1.366,1.366) {12};
  \node[vertex] (v3) at (-1.866,2.232) {3};
  \node[vertex] (v8) at (1.366,-0.366) {8};
  \node[vertex] (v13) at (1.732,0) {13};
  \node[vertex] (v16) at (1.866,0.5) {16};
  \node[vertex] (v18) at (1.732,1) {18};
  \node[vertex] (v19) at (1.366,1.366) {19};
  \node[vertex] (v11) at (2.598,-0.5) {11};
  \node[vertex] (v4) at (3.464,-1) {4};
  \node[vertex] (v9) at (2.832,0.241) {9};
  \node[vertex] (v10) at (2.832,0.759) {10};
  \node[vertex] (v1) at (2.732,1) {1};
  \node[vertex] (v6) at (2.672,1.342) {6};
  \node[vertex] (v7) at (2.498,1.643) {7};
  \node[vertex] (v15) at (2.232,1.866) {15};
  \node[vertex] (v2) at (3.612,1.684) {2};
  \foreach \a/\b in {21/14,21/17,21/20,14/5,14/12,12/3,20/8,20/13,20/16,20/18,20/19,13/11,11/4,16/9,16/10,18/1,18/6,18/7,18/15,6/2}
    \draw[edge] (v\a)--(v\b);
\end{tikzpicture}
\par\vspace{0.8em}
\begin{tikzpicture}[x=1.35cm,y=1.35cm,vertex/.style={circle,draw,fill=white,minimum size=6mm,inner sep=0pt,font=\small},edge/.style={-{Latex[length=2mm]},semithick}]
  \node[font=\small] at (0.35,2.0) {(b) Outward orientation};
  \node[vertex] (a4) at (0,0) {4};
  \node[vertex] (a2) at (-0.707,0.707) {2};
  \node[vertex] (a3) at (0.707,0.707) {3};
  \node[vertex] (a1) at (1.414,1.414) {1};
  \node[vertex] (a5) at (0,-1) {5};
  \draw[edge] (a3)--(a1);
  \draw[edge] (a4)--(a2);
  \draw[edge] (a4)--(a3);
  \draw[edge] (a5)--(a4);
\end{tikzpicture}
\hspace{1.4cm}
\begin{tikzpicture}[x=1.35cm,y=1.35cm,vertex/.style={circle,draw,fill=white,minimum size=6mm,inner sep=0pt,font=\small},edge/.style={-{Latex[length=2mm]},semithick}]
  \node[font=\small] at (0.4,1.6) {(c) Merging orientation};
  \node[vertex] (b4) at (0,0) {4};
  \node[vertex] (b2) at (-0.866,0.5) {2};
  \node[vertex] (b3) at (0.866,0.5) {3};
  \node[vertex] (b1) at (1.732,1) {1};
  \node[vertex] (b5) at (0,-1) {5};
  \draw[edge] (b3)--(b1);
  \draw[edge] (b2)--(b4);
  \draw[edge] (b3)--(b4);
  \draw[edge] (b5)--(b4);
\end{tikzpicture}
\hspace{1.4cm}
\begin{tikzpicture}[x=1.35cm,y=1.35cm,vertex/.style={circle,draw,fill=white,minimum size=6mm,inner sep=0pt,font=\small},edge/.style={-{Latex[length=2mm]},semithick}]
  \node[font=\small] at (0,2.0) {(d) Mixed path orientation};
  \node[vertex] (c1) at (1.414,1.414) {1};
  \node[vertex] (c2) at (0.707,0.707) {2};
  \node[vertex] (c3) at (0,0) {3};
  \node[vertex] (c4) at (-0.707,-0.707) {4};
  \node[vertex] (c5) at (-1.414,-1.414) {5};
  \draw[edge] (c1)--(c2);
  \draw[edge] (c2)--(c3);
  \draw[edge] (c4)--(c3);
  \draw[edge] (c5)--(c4);
\end{tikzpicture}
\caption{Directed-tree topologies used in the comparison with Adlercreutz and Pates~\cite{AdlercreutzPates2026}, redrawn here so that the numerical examples are self-contained.  Panel (a) corresponds to their Fig.~6, while panels (b)--(d) correspond to their Figs.~7(a)--7(c), respectively.}
\label{fig:comparison-graphs}
\end{figure*}
\DIFaddend

For the four-storage path in \eqref{eq:four-storage-sr-data} with $\beta=1/2$, \REV{a compatible LI pair}, after changing from the interleaved Adlercreutz--Pates state ordering to our storage--delay ordering, \REV{is}
\[
    K_1=
    \begin{bmatrix}
        3/2&0&0\\
        -1/2&7/6&0\\
        0&-1/2&15/14
    \end{bmatrix}.
\]
{\small
\[
    K_2=\REV{-}\frac{1}{2\sqrt2}
    \begin{bmatrix}
        1&-2&0&0&1&-2&0\\
        0&1&-2&0&0&1&-2\\
        0&0&1&-2&0&0&1
    \end{bmatrix}.
\]
}
The corresponding compressed input factor is \eqref{eq:four-storage-sr-FZ}; equivalently, the online factors are $Y$ and $U=[FZ\;0]$.  For the remaining four examples we report only nonzero patterns and nonzero counts.

\REV{The LI counts come from \cite{AdlercreutzPates2026}; the RI counts use the displayed edges.}
\begin{table}[!t]
\REV{\caption{Factor nonzero counts and graph bounds for the five directed trees at $\beta=1/2$.}\label{tab:factor-counts}}
\centering
{\scriptsize
\[
\DIFdelbegin 
\DIFdelend \DIFaddbegin \begin{array}{c|cc|ccc}
\text{example} & \operatorname{nnz}(K_1) & \operatorname{nnz}(K_2)
& \operatorname{nnz}(Y) & \operatorname{ub}_{\rm graph} & \operatorname{nnz}(U)\\
\hline
\text{path} & 5 & 11 & \REV{16} & 8 & 8\\
\text{large tree} & 79 & 77 & \REV{101} & \REV{99} & 57\\
\text{outward tree} & 9 & 15 & \REV{21} & \REV{13} & \REV{11}\\
\text{merging tree} & 12 & 18 & \REV{21} & \REV{20} & \REV{8}\\
\text{mixed path} & 10 & 16 & \REV{21} & \REV{16} & \REV{12}
\end{array}\DIFaddend 
\]
}
\end{table}
\REV{In Table~\ref{tab:factor-counts},} \DIFaddbegin \DIFadd{$\operatorname{ub}_{\rm graph}$ denotes the right-hand side of }\DIFaddend \eqref{eq:graph-mask-U-bound}.  \DIFaddbegin \DIFadd{By Theorem~\ref{thm:explicit-height-solution}, this bound holds under Assumption~\ref{ass:data} without an additional hypothesis on $H$.  }\DIFaddend \REV{The blocks of $Y$ in \eqref{eq:sr-known-H-Y} give $\operatorname{nnz}(Y)=N+4m$. The bound is tight only for the path. For the merging tree, $S^\top R=0$, hence $FC=rD^{-1}(S^\top R-D)=-rI_4$ and $\operatorname{nnz}(U)=8$, while $\operatorname{ub}_{\rm graph}=20$. The transitive-closure mask permits entries that cancel. Across the five examples, the reduced-DARE residual is below $5\times10^{-16}$, with no support-mask violations.}  The comparison is not monotone in the factor counts: $U$ is sparser than $K_2$ in all five examples, but $Y$ can have more nonzeros than $K_1$.

\REV{For complete online evaluation, we count scalar additions/subtractions, multiplications, and divisions. Each nonempty dot product counts one multiplication per stored term and one fewer additions. LI and RI use the lowest-operation tested direct or precomputed block evaluations; RI also uses tree sums. On merging trees, both input evaluations use diagonal-plus-rank-one blocks. The path has edges $(i+1)\to i$ for $1\le i<N$; balanced outward and merging trees have edges $\lfloor i/2\rfloor\to i$ and $i\to\lfloor i/2\rfloor$, respectively.}
\begin{table}[!t]
\REV{\caption{Online scalar operations at $\beta=1/2$.}\label{tab:online-li-ri}}
\centering
\small
\REV{\begin{tabular}{lcc}
\hline
Tree & $N=80$, LI/RI & $N=160$, LI/RI\\
\hline
Path & $708/790$ & $1428/1590$\\
Balanced outward & $943/788$ & $1904/1588$\\
Balanced merging & $1017/909$ & $2057/1829$\\
\hline
\end{tabular}}
\end{table}
\REV{In Table~\ref{tab:online-li-ri}, RI uses about $11$--$17\%$ fewer counted operations on the balanced trees; LI uses about $10\%$ fewer on paths. These are full online solve-and-product arithmetic counts for the specified schedules, excluding indexing and memory traffic. They do not include offline construction or establish runtime, memory-use, or communication advantages. The selected schedules agree with feedback from the explicit reduced Riccati solution within $10^{-9}$ per entry.}

\REV{For each $N\in\{10,20,30,40,50\}$, we sampled $20$ uniformly random labeled trees and roots, orienting each tree away from its root (seed $20261117$, $\beta=1/2$). Entries below $10^{-10}\max\{1,\max|U_{ij}|\}$ were treated as zero. The LI and RI operation counts use the lowest-operation valid tested schedule for each sample under the convention of Table~\ref{tab:online-li-ri}.}
\begin{table}[!t]
\REV{\caption{Mean online operations, factor nonzero counts, and graph bounds for 20 random outward trees per size.}\label{tab:random-counts}}
\centering
\scriptsize
\[
\begin{array}{c|cc|ccc}
N & \REV{\text{LI ops.}} & \REV{\text{RI ops.}} & \REV{\operatorname{nnz}(K_2)} & \REV{\operatorname{nnz}(U)} & \REV{\operatorname{ub}_{\rm graph}}\\
\hline
10 & \REV{95.50} & \REV{88.70} & \REV{34.35} & \REV{25.35} & \REV{32.95}\\
20 & \REV{209.50} & \REV{187.10} & \REV{73.55} & \REV{54.55} & \REV{72.85}\\
30 & \REV{325.25} & \REV{288.50} & \REV{114.25} & \REV{85.25} & \REV{111.25}\\
40 & \REV{443.00} & \REV{388.50} & \REV{154.25} & \REV{115.25} & \REV{152.65}\\
50 & \REV{562.55} & \REV{487.70} & \REV{193.85} & \REV{144.85} & \REV{193.65}
\end{array}
\]
\end{table}
\FloatBarrier
\REV{In the $100$ sampled trees, $U$ has fewer nonzeros than $K_2$ in every case, the tested RI schedule uses fewer online scalar operations in $99$ cases, and the graph bound holds but is tight in none.}

\section{Conclusions}

\DIFaddbegin
\DIFadd{We studied optimal feedback for sender--receiver transportation linear-quadratic control on directed trees.  The optimal control signal can be computed by a sparse linear solve followed by a sparse input multiplication, without imposing communication or decentralization constraints.  For every orientation of a connected tree, the stabilizing reduced Riccati solution has an explicit formula.  Its incidence products yield a support bound determined by the directed tree and the discount factor.

\REV{For the tested schedules, tree-sum right-inverse evaluation uses fewer counted scalar operations for the online solve and product on balanced trees, while left-inverse evaluation uses fewer on paths.}  Future work should treat cyclic graphs, \DIFdel{more general costs, }\REV{the possible use of the graph-derived input-factor zeros when optimizing with an input penalty,} and direct numerical construction of the sparse factors.}
\DIFaddend

\appendix[Proofs]\label{app:proofs}

\begin{proof}[Proof of Lemma~\ref{lem:tree-incidence-rank-Z}]
A connected tree with $N$ vertices has $N-1$ edges, so $m=N-1$.  For each edge $e$, the corresponding column of $C$ has one entry $+1$, one entry $-1$, and all other entries zero.  Therefore $\ones^\top C=0$, so $\rank C\le N-1$.
It remains to prove $\rank C\ge N-1$.  Suppose $y\in\R^N$ satisfies $C^\top y=0$.  For an edge $e:s\to r$, the $e$th component of $C^\top y$ is \(\hat e_r^\top y-\hat e_s^\top y=y_r-y_s\).
Thus $C^\top y=0$ implies that $y_r=y_s$ for every directed edge $s\to r$.  Since the underlying undirected graph is connected, equality propagates along undirected paths, and all components of $y$ are equal.  Hence $\ker C^\top=\operatorname{span}\{\ones\}$, so $\rank C=N-1$ by the rank-nullity theorem.
To prove that $Z$ is nonsingular, suppose $Z\begin{bmatrix}\alpha^\top& \gamma\end{bmatrix}^\top=0$ with $\alpha\in\R^m$ and $\gamma\in\R$.  Then \(C\alpha+\gamma\ones=0\).
Multiplying by $\ones^\top$ gives \(\ones^\top C\alpha+\gamma\ones^\top\ones=\gamma N=0\),
so $\gamma=0$.  Then $C\alpha=0$.  Since $C$ has $m=N-1$ columns and $\rank C=N-1$, its columns are linearly independent, and $\alpha=0$.  Therefore $\ker Z=\{0\}$, and the square matrix $Z$ is nonsingular.
\end{proof}

\begin{lemma}\label{lem:tree-edge-column-rank}
Let $G$ satisfy Assumption~\ref{ass:data}.  Fix $z\in\mathbb C$ with $|z|\ge1$ and $0<r<1$.  Let $M_z\in\mathbb C^{N\times m}$ have column $r\hat e_b-z\hat e_a$ for each edge $e_i:a\to b$.  Then $M_z$ has full column rank.
\end{lemma}

\begin{proof}
Suppose $M_zy=0$.  Since $G$ is a tree by Assumption~\ref{ass:data}, it has a leaf $v$.  Let $e_i$ be the unique edge incident to $v$.  In row $v$ of $M_zy=0$, the only nonzero term comes from column $i$.  Its coefficient is $r$ if $e_i$ enters $v$, and it is $-z$ if $e_i$ leaves $v$.  Both coefficients are nonzero because $r>0$ and $|z|\ge1$.  Hence $y_i=0$.
Remove $v$ and $e_i$.  For every remaining vertex, the equation in $M_zy=0$ differs from the corresponding equation for the smaller directed tree only by the contribution of $e_i$, and that contribution is zero because $y_i=0$.  Therefore the remaining entries of $y$ satisfy the same equation for the smaller directed tree.  Repeating this argument removes one edge at each step.  Induction on $m$ gives $y=0$, so $M_z$ has full column rank.
\end{proof}

\begin{proof}[Proof of Lemma~\ref{lem:outward-rooted-orientation}]
For each $i$, column $i$ of $R$ is the standard basis vector at the receiving vertex of edge $e_i$.  If two different directed edges had the same receiving vertex, then the corresponding two columns of $R$ would be equal.  Equal columns are linearly dependent, contradicting Assumption~\ref{ass:outward-rooted}.  Hence no two directed edges have the same receiving vertex.
By Lemma~\ref{lem:tree-incidence-rank-Z}, the graph has $m=N-1$ edges.  Since no two directed edges have the same receiving vertex, exactly $N-1$ vertices receive an edge.  Hence exactly one vertex receives no edge; this is $v_{\rm root}$ in the lemma statement.
It remains to prove that every edge is directed away from $v_{\rm root}$.  Let $v\ne v_{\rm root}$.  Since $G$ is connected by Assumption~\ref{ass:data}, there is an undirected path from $v_{\rm root}$ to $v$.  Write this path as \(v_{\rm root}=v_0,\ v_1,\ldots,\ v_\ell=v\).
We prove by induction on $j$ that the edge between $v_{j-1}$ and $v_j$ is directed $v_{j-1}\to v_j$.  For $j=1$, the edge cannot be directed $v_1\to v_{\rm root}$, because $v_{\rm root}$ receives no edge.  Therefore it is directed $v_{\rm root}\to v_1$.
Assume that the edge between $v_{j-2}$ and $v_{j-1}$ is directed $v_{j-2}\to v_{j-1}$ for some $2\le j\le\ell$.  Then $v_{j-1}$ already receives the edge from $v_{j-2}$.  The edge between $v_{j-1}$ and $v_j$ cannot be directed $v_j\to v_{j-1}$, because no two directed edges have the same receiving vertex.  Therefore it is directed $v_{j-1}\to v_j$.
The induction shows that each edge on the path from $v_{\rm root}$ to $v$ is directed away from $v_{\rm root}$.  Since $v\ne v_{\rm root}$ was arbitrary and every edge lies on the path from $v_{\rm root}$ to one of its endpoints, every edge of $G$ is directed away from $v_{\rm root}$.
\end{proof}

\begin{proof}[Proof of Theorem~\ref{thm:sender-receiver-dare-unique}]
We apply Hansson and Hagander~\cite[Theorem~1]{HanssonHagander1999} to \eqref{eq:dare} with the stage-cost matrix $Q$ in Section~\ref{sec:problem}.  The block matrix required by Hansson and Hagander~\cite[Theorem~1]{HanssonHagander1999} is positive semidefinite because $Q\succeq 0$.
The pair $(A,B)$ is stabilizable because $A$ is block upper triangular with diagonal blocks $rI$ and $0$, and $0<\beta<1$ with $r=\sqrt\beta$ gives $0<r<1$.
It remains to verify the rank condition in Hansson and Hagander~\cite[Theorem~1]{HanssonHagander1999}.  For every $|z|\ge1$, the matrix
\[
    \begin{bmatrix}
        A-zI&B\\
        Q&0
    \end{bmatrix}
\]
must have full column rank.  Let $x=(s,d)$, define $f_z(s,d,u)=(r(s+Rd)-zs-Su,-zd+u)$ and $g(s,d,u)=(s,0)$, and suppose the product of the displayed matrix with $(x,u)$ is zero.  Then $f_z(s,d,u)=0$ and $g(s,d,u)=0$.  The equation $g(s,d,u)=0$ gives $s=0$.  Substituting $s=0$ into $f_z(s,d,u)=0$ gives
$(rR-zS)d=0$ and $u=zd$.
The matrix $rR-zS$ has column $r\hat e_b-z\hat e_a$ for each edge $e_i:a\to b$.  Lemma~\ref{lem:tree-edge-column-rank} gives $d=0$.  Hence $u=zd=0$ and $s=0$, so $x=0$.  Thus the rank condition in Hansson and Hagander~\cite[Theorem~1]{HanssonHagander1999} holds.  Hansson and Hagander~\cite[Theorem~1]{HanssonHagander1999} give existence and uniqueness of a positive semidefinite stabilizing solution of \eqref{eq:dare}.
\end{proof}

\begin{proof}[Proof of Lemma~\ref{lem:full-reduced-dare-equivalence}]
First suppose that $P$ satisfies \eqref{eq:dare} and $D(P)$ is nonsingular.  The derivation of \eqref{eq:sr-P-from-H} from \eqref{eq:dare} gives \eqref{eq:sr-P-block-form}.  Substituting \eqref{eq:sr-P-block-form} into \eqref{eq:sr-DP}--\eqref{eq:sr-H-of-P} gives
\begin{align*}
    D(P)&=D, & N(P)&=N, \\
    H(P)&=I+\beta H-ND^{-1}N^\top .
\end{align*}
Since \eqref{eq:sr-P-block-form} also gives $H(P)=H$, the reduced DARE \eqref{eq:sr-reduced-H-DARE} follows.
Conversely, suppose that $H$ satisfies \eqref{eq:sr-reduced-H-DARE-expanded} and define $P$ by \eqref{eq:sr-P-block-form}.  Substituting \eqref{eq:sr-P-block-form} into \eqref{eq:sr-DP}--\eqref{eq:sr-H-of-P} gives $D(P)=D$, $N(P)=N$, and
\[
    H(P)=I+\beta H-ND^{-1}N^\top .
\]
Using \eqref{eq:sr-reduced-H-DARE} gives $H(P)=H$.  Therefore \eqref{eq:sr-P-block-form} is the same as \eqref{eq:sr-P-from-H}.  Since \eqref{eq:dare} is equivalent to \eqref{eq:sr-P-from-H} under the definitions \eqref{eq:sr-DP}--\eqref{eq:sr-H-of-P}, the matrix $P$ satisfies \eqref{eq:dare}.

It remains to compare the two stabilizing conditions.  The feedback associated with \eqref{eq:dare} is $u=-rD^{-1}N^\top(s+Rd)$.
Put $w=s+Rd$.  From \eqref{eq:sr-dynamics} $s^+=rw-Su$ and $\qquad d^+=u$.
Thus
\begin{align*}
    w^+&=s^++Rd^+ =rw+(R-S)u \\
       &=r\bigl(I-CD^{-1}N^\top\bigr)w=(rI-rCD^{-1}N^\top)w,
\end{align*}
where the third equality uses $C=R-S$ and the displayed feedback.  Also
$d^+=-rD^{-1}N^\top w$.
In the coordinates $(w,d)$, the full closed-loop matrix is therefore
\[
    \begin{bmatrix}
        r(I-CD^{-1}N^\top)&0\\
        -rD^{-1}N^\top&0
    \end{bmatrix}.
\]
The coordinate change from $(s,d)$ to $(w,d)$ is nonsingular because $w=s+Rd$.  Hence the eigenvalues of the original closed-loop matrix are the eigenvalues of \eqref{eq:sr-reduced-closed-loop} together with $m$ zeros.  The original closed-loop matrix is stable if and only if every eigenvalue of \eqref{eq:sr-reduced-closed-loop} lies in the open unit disk.
\end{proof}

\begin{proof}[Proof of Lemma~\ref{lem:reduced-dare-existence-conditions}]
Consider the semidefinite DARE
{\small
\begin{equation}\label{eq:reduced-Hansson-Hagander-DARE}
\begin{aligned}
    H={}&A_r^\top HA_r+Q_r 
      -(A_r^\top HB_r+L_r)\\
      &\quad\times(R_r+B_r^\top HB_r)^{-1}
      (B_r^\top HA_r+L_r^\top).
\end{aligned}
\end{equation}
}
with
\[
\begin{gathered}
    A_r=rI,\qquad B_r=rC,\qquad Q_r=I,\\
    L_r=-S,\qquad R_r=S^\top S .
\end{gathered}
\]
Since $r^2=\beta$,
\begin{align*}
    A_r^\top HA_r&=\beta H,\\
    R_r+B_r^\top HB_r&=S^\top S+\beta C^\top HC,\\
    A_r^\top HB_r+L_r&=\beta HC-S .
\end{align*}
Substituting these identities into \eqref{eq:reduced-Hansson-Hagander-DARE} gives \eqref{eq:sr-reduced-H-DARE-expanded}.  The cost matrix is positive semidefinite because
\[
    \begin{bmatrix}Q_r&L_r\\L_r^\top&R_r\end{bmatrix}
    =
    \begin{bmatrix}I\\-S^\top\end{bmatrix}
    \begin{bmatrix}I&-S\end{bmatrix}
    \succeq0 .
\]
Since $0<r<1$, every eigenvalue of $A_r$ lies strictly inside the unit disk.  Hence $(A_r,B_r)$ is stabilizable.
It remains to verify the rank hypothesis in Theorem~1 of Hansson and Hagander~\cite[Theorem~1]{HanssonHagander1999}.  For $|z|\ge1$, consider
\[
    \begin{bmatrix}A_r-zI&B_r\\ I&-S\end{bmatrix}.
\]
Suppose \((A_r-zI)x+B_ru=0, \qquad x-Su=0\).
Substituting $x=Su$ and using $A_r=rI$ and $B_r=rC$ gives
\begin{equation}\label{eq:reduced-rank-condition-edge-matrix}
    (rC+(r-z)S)u=0.
\end{equation}
For each edge $e_i:s\to t$, column $i$ of $rC+(r-z)S$ is $r\hat e_t-z\hat e_s$.  Lemma~\ref{lem:tree-edge-column-rank} applied to \eqref{eq:reduced-rank-condition-edge-matrix} gives $u=0$.  Then $x=Su=0$.  The displayed matrix has full column rank for every $|z|\ge1$.
The cost matrix is positive semidefinite, $(A_r,B_r)$ is stabilizable, and the rank hypothesis in Theorem~1 of Hansson and Hagander holds.  Theorem~1 of Hansson and Hagander~\cite[Theorem~1]{HanssonHagander1999} therefore gives the stated unique stabilizing positive semidefinite solution.  The same theorem gives nonsingularity of $R_r+B_r^\top HB_r$ for this solution.  The identity $R_r+B_r^\top HB_r=S^\top S+\beta C^\top HC$ shows that this matrix is $D$ in \eqref{eq:sr-DH}.
\end{proof}

\DIFdelbegin 
\DIFdel{For one term $q_\Omega\xi_\Omega\xi_\Omega^\top$ in \eqref{eq:downstream-H-ansatz},
}
    \DIFdel{q_\Omega\xi_\Omega\xi_\Omega^\top c_i
    =q_\Omega\xi_\Omega(1_{r_i\in \Omega}-1_{s_i\in \Omega}).
}
\DIFdel{Let $e:s\to r$ be any directed edge and let $v\in V_{\rm rec}$.  If $e\notin E_v^{\rm in}$, then $e$ remains after deleting the edges in $E_v^{\rm in}$.  Hence $s$ }\DIFdelend \DIFaddbegin
\begingroup
\color{black}
\def\DIFaddend{\color{black}}
\def\DIFdelend{\color{black}}
\def\DIFaddendFL{\color{black}}
\def\DIFdelendFL{\color{black}}
\begin{proof}[Proof of Proposition~\ref{prop:ri-factorization}]
\DIFadd{Lemma~\ref{lem:tree-incidence-rank-Z} makes $Z$ in \eqref{eq:Z-definition} nonsingular.  Hence the block triangular matrix $Y$ in \eqref{eq:sr-known-H-Y} is nonsingular.  Given $x_k=(s_k,d_k)$, equations \eqref{eq:sr-known-H-variables} give $\delta_k=d_k$ and \eqref{eq:sr-known-H-w-representation} gives $Z\zeta_k=s_k+Rd_k$.  Equations \eqref{eq:sr-feedback-from-H-state}, \eqref{eq:sr-known-H-F}, }\DIFaddend and \DIFdelbegin \DIFdel{$r$ are in the same connected component after this deletion, so $s\in{\mathcal R}(v)$ implies $r\in{\mathcal R}(v)$.  If $e\in E_v^{\rm in}$, then $r=v$.  Since ${\mathcal R}(v)$ is the connected component containing $v$ after deleting the edges in $E_v^{\rm in}$, $r\in{\mathcal R}(v)$ also in this case.  Therefore
}\DIFdel{
    s\in{\mathcal R}(v)\quad\Longrightarrow\quad r\in{\mathcal R}(v),
    \qquad v\in V_{\rm rec}.
}
\DIFdel{If $\Omega\in{\mathcal L}$ and $s\in\Omega$, then either $\Omega=V$ or \eqref{eq:component-lattice} gives $\Omega=\bigcap_{v\in J}{\mathcal R}(v)$ for a nonempty set $J\subseteq V_{\rm rec}$.  In the first case $r\in\Omega$.  In the second case, $s\in{\mathcal R}(v)$ for every $v\in J$, so \eqref{eq:receiver-side-edge-closure} gives $r\in{\mathcal R}(v)$ for every $v\in J$, and hence $r\in\Omega$.  Thus , for every $\Omega\in{\mathcal L}$, the alternative $s\in\Omega$ and $r\notin\Omega$ is impossible.
Applying this to $e_i:s_i\to r_i$ shows that $1_{r_i\in\Omega}-1_{s_i\in\Omega}$ is nonzero precisely when $r_i\in\Omega$ and $s_i\notin\Omega$.  Summing over $\Omega$ proves \eqref{eq:downstream-HC-identity}.  Multiplying by $c_i^\top=(\hat e_{r_i}-\hat e_{s_i})^\top$ gives \eqref{eq:downstream-CtHC-identity}.  Equation~\eqref{eq:downstream-CtHC-identity} also shows that $c_i^\top Hc_j$ can be nonzero only if some set $\Omega\in{\mathcal L}$ separates the endpoints of both $e_i$ and $e_j$.  }\DIFdelend \DIFaddbegin \DIFadd{\eqref{eq:sr-known-H-U} then give
}\[
    \DIFadd{U\eta_k=FZ\zeta_k=F(s_k+Rd_k)=u_k.
}\]
\DIFadd{Thus $u_k=UY^{-1}x_k=-Kx_k$, so $K=-UY^{-1}$.
}\DIFaddend \end{proof}

\DIFdelbegin 
\DIFdel{It is enough to prove \eqref{eq:atom-contrast-lifting} for
\(z=\hat e_p-\hat e_q\), where $p$ and $q$ belong to the same class $A_\alpha$ in \eqref{eq:boolean-atom-relation}.  Every vector in \eqref{eq:boolean-atom-contrast-set} is a linear combination of such vectors.
Let the }\DIFdelend \DIFaddbegin \begin{proof}[Proof of Lemma~\ref{lem:tree-height}]
\DIFadd{Fix $v_0\in V$ and set $\widetilde h(v_0)=0$.  For $v\in V$, traverse the }\DIFaddend unique undirected path from \DIFdelbegin \DIFdel{$p$ to $q$ be
$v_0=p,\ v_1,\ldots,\ v_k=q$.  For each path edge, define
}
    \DIFdel{\sigma_h=\begin{cases}
    1, & \text{if the directed edge is }v_h\to v_{h+1},\\
    -1, & \text{if the directed edge is }v_{h+1}\to v_h,
    \end{cases}
}
\DIFdelend \DIFaddbegin \DIFadd{$v_0$ to $v$, adding $1$ when an edge is traversed in its directed orientation and subtracting $1$ otherwise.  The path is unique by Assumption~\ref{ass:data}, so $\widetilde h$ is well defined.  For every edge $e_i:s_i\to t_i$, the path construction gives
}\DIFadd{\REV{\(\widetilde h(t_i)-\widetilde h(s_i)=1\).}}\ \DIFaddend 
\DIFdelbegin \DIFdel{$h=0,\ldots,k-1$.  }\DIFdelend \DIFaddbegin \DIFadd{Set $h(v)=\widetilde h(v)-\min_{v'\in V}\widetilde h(v')$.  Then $h$ satisfies \eqref{eq:tree-height} and has minimum zero.  }\DIFaddend If \DIFdelbegin \DIFdel{$\sigma_0=-1$, then the edge between $v_0$ and $v_1$ enters $p$.  Deleting all edges entering $p$ puts $p$ in ${\mathcal R}(p)$ and puts $q$ outside ${\mathcal R}(p)$, because the deleted edgeis the first edge on the unique path from $p$ to $q$.  This contradicts $p\sim_{\mathcal B}q$ in \eqref{eq:boolean-atom-relation}.  Hence $\sigma_0=1$.  The same argument at $q$ gives $\sigma_{k-1}=-1$, because $\sigma_{k-1}=1$ means that the last path edge enters $q$ when the path is read from $p$ to $q$.
For an internal vertex $v_h$, $1\le h\le k-1$, the case $\sigma_{h-1}=\sigma_h$ means that exactly one of the two path edges incident with $v_h$ enters $v_h$.  Deleting the edges entering $v_h$ removes that path edge.  The remaining path edge is not deleted, so ${\mathcal R}(v_h)$ contains $v_h$ and one of the two endpoints $p,q$, but not the other endpoint.  This contradicts $p\sim_{\mathcal B}q$.  Thus
}\DIFdel{
    \sigma_h=-\sigma_{h-1},
    \qquad h=1,\ldots,k-1 .
}
\DIFdelend \DIFaddbegin \DIFadd{$h_1$ and $h_2$ have these properties, then $h_1-h_2$ has equal values at the endpoints of every edge.  Connectivity gives $h_1-h_2=\gamma\ones$ for a scalar $\gamma$.  The two minimum-zero normalizations give $\gamma=0$, so $h_1=h_2$.
}\end{proof}

\begin{proof}[Proof of Theorem~\ref{thm:explicit-height-solution}]
\DIFadd{\REV{The proof has four steps.}  First, we express the reduced problem as an infinite-horizon flow problem.  Second, height-component balance identities give a componentwise lower bound on its value.  Third, incidence equations on the component trees construct an input sequence attaining every lower bound simultaneously.  Fourth, the resulting quadratic value is shown to satisfy the reduced DARE and its closed-loop matrix is shown to be stable.
}
\DIFadd{Consider the following infinite-horizon optimization problem:
}\begin{equation}\DIFadd{\label{eq:proof-reduced-value}
\begin{aligned}
    {\mathcal J}(w)&=\inf_{(u_t)_{t\ge0}}
    \sum_{t=0}^{\infty}\|z_t-Su_t\|^2,\\
    z_0&=w,
    \qquad
    z_{t+1}=r(z_t+Cu_t).
\end{aligned}
}\end{equation}\DIFaddend 
\DIFdelbegin \DIFdel{Let $y$ be zero on all edges outside this path.
On the path edge between $v_h$ and $v_{h+1}$, set
}\DIFdel{
    y_h=\sigma_h .
}
\DIFdelend \DIFaddbegin \DIFadd{We compute its value and then verify the reduced DARE \eqref{eq:sr-reduced-H-DARE-expanded} directly.
}

\DIFadd{Set
}\begin{equation}\DIFadd{\label{eq:proof-scaled-variables}
    \bar z_t=r^{-t}z_t,
    \qquad
    v_t=r^{-t}u_t,
    \qquad
    y_t=\bar z_t-Sv_t.
}\end{equation}\DIFaddend 
\DIFdelbegin \DIFdel{The matrix $S$ records the sender of each edge.  From $\sigma_0=1$, the first path edge is sent by $p$, so $(Sy)_p=1$.  From $\sigma_{k-1}=-1$, the last path edge is sent by $q$ with coefficient $-1$, so $(Sy)_q=-1$.  For $1\le h\le k-1$, the two path-edge contributions to $(Sy)_{v_h}$ are present exactly when
}\DIFdelend \DIFaddbegin \DIFadd{Using $C=R-S$ in \eqref{eq:proof-reduced-value} gives
}\begin{equation}\DIFadd{\label{eq:proof-flow-identities}
    \bar z_{t+1}=\bar z_t+Cv_t=y_t+Rv_t,
    \qquad
    \bar z_t=y_t+Sv_t.
}\end{equation}
\DIFadd{Moreover,
}\begin{equation}\DIFadd{\label{eq:proof-scaled-cost}
    \sum_{t=0}^{\infty}\|z_t-Su_t\|^2
    =\sum_{t=0}^{\infty}\beta^t\|y_t\|^2.
}\end{equation}
\DIFadd{For vertex $i$, }\DIFaddend the \DIFdelbegin \DIFdel{corresponding edge is sent by $v_h$.  If both adjacent path edges are received by $v_h$, then both contributions are zero.  If both adjacent path edges are sent by $v_h$, then \eqref{eq:path-sign-alternation} and \eqref{eq:path-lift-y-definition} make the two contributions sum to zero.  Therefore
}\DIFdel{
    Sy=\hat e_p-\hat e_q.
}
\DIFdelend \DIFaddbegin \DIFadd{two identities in \eqref{eq:proof-flow-identities} give
}\begin{align}
    \DIFadd{w_i}&\DIFadd{=y_{0,i}+\sum_{e:s(e)=i}v_{0,e},
        \label{eq:proof-balance-zero}}\\
    \DIFadd{y_{t-1,i}+\sum_{e:t(e)=i}v_{t-1,e}
    }&\DIFadd{=y_{t,i}+\sum_{e:s(e)=i}v_{t,e},
    \qquad t\ge1,
        \label{eq:proof-balance-positive}
}\end{align}\DIFaddend 
\DIFdelbegin \DIFdel{The vector $Cy$ has a contribution $-\sigma_h$ at $v_h$ }\DIFdelend \DIFaddbegin \DIFadd{where $s(e)$ }\DIFaddend and \DIFdelbegin \DIFdel{$\sigma_h$ at $v_{h+1}$ from the path edge between them.  Let $A_\gamma$ be any class in \eqref{eq:boolean-atom-relation}.  Contributions from path edges with both endpoints in $A_\gamma$ cancel in \(\sum_{i\in A_\gamma}(Cy)_i\).  For each maximal consecutive subpath with vertices in $A_\gamma$, \eqref{eq:path-sign-alternation} makes the two boundary contributions cancel.  If such a subpath contains $p$ or $q$, then the endpoint contribution cancels with the single boundary contribution because $\sigma_0=1$ }\DIFdelend \DIFaddbegin \DIFadd{$t(e)$ denote the sender and receiver of edge $e$ within this proof.
}

\DIFadd{Fix $a\le h_{\max}$ and $\Omega\in{\mathcal C}_a$.  Sum \eqref{eq:proof-balance-zero}--\eqref{eq:proof-balance-positive} over
}\DIFadd{\REV{\(\{(v,t):v\in\Omega,\ 0\le t\le h(v)-a\}\).}}\ 
\DIFadd{For each vertex, the vertical terms telescope from $w_v$ to $y_{h(v)-a,v}$.  A transport variable $v_{t,e}$ with edge $e:s\to q$ has its sender occurrence in the sum exactly when
}\DIFadd{\REV{\(s\in\Omega\) and \(t\le h(s)-a\).}}\ 
\DIFadd{Then $h(q)=h(s)+1\ge a$, the edge belongs to $G[V_a]$, }\DIFaddend and \DIFdelbegin \DIFdel{$\sigma_{k-1}=-1$.  Hence
$\sum_{i\in A_\gamma}(Cy)_i=0,
    \; \gamma=1,\ldots,\ell$.
By \eqref{eq:boolean-atom-contrast-set}, $Cy\in{\mathcal W}_{\mathcal B}$.  Linearity gives
\eqref{eq:atom-contrast-lifting} for every $z\in{\mathcal W}_{\mathcal B}$.
}
\DIFdelend \DIFaddbegin \DIFadd{$q$ lies in the same component $\Omega$.  Also
}\DIFadd{\REV{\(t+1\le h(s)-a+1=h(q)-a\),}}\ 
\DIFadd{so the receiver occurrence is included.  The converse implication follows by reversing these equalities.  Hence all transport terms cancel and
}\begin{equation}\DIFadd{\label{eq:proof-cut-identity}
    \sum_{v\in\Omega}y_{h(v)-a,v}
    =\sum_{v\in\Omega}w_v.
}\end{equation}
\DIFaddend 

\DIFdelbegin 
\DIFdel{By the definition of $C$ in Section~\ref{sec:problem}
$c_i=\hat e_{r_i}-\hat e_{s_i}$.  For one term in \eqref{eq:general-indicator-span-matrix}, }\DIFdelend \DIFaddbegin \DIFadd{Every pair $(v,t)\in V\times\mathbb Z_{\ge0}$ determines the unique level $a=h(v)-t$ and the unique component $\Omega\in{\mathcal C}_a$ containing $v$.  Conversely, $a\le h(v)$ gives the nonnegative time $t=h(v)-a$.  Regrouping the nonnegative terms in \eqref{eq:proof-scaled-cost} therefore gives
}\begin{equation}\DIFadd{\label{eq:proof-energy-partition}
\sum_{t=0}^{\infty}\beta^t\|y_t\|^2
=
\sum_{a=-\infty}^{h_{\max}}
\sum_{\Omega\in{\mathcal C}_a}
\sum_{v\in\Omega}
\beta^{h(v)-a}y_{h(v)-a,v}^2.
}\end{equation}

\DIFadd{For a scalar $\iota$, weighted Cauchy--Schwarz gives
}\begin{align*}
\DIFadd{\iota^2
}&\DIFadd{=\left(\sum_{v\in\Omega}
\beta^{(h(v)-a)/2}f_v
\beta^{(a-h(v))/2}\right)^2}\\
&\DIFadd{\le
\left(\sum_{v\in\Omega}\beta^{h(v)-a}f_v^2\right)
\left(\sum_{v\in\Omega}\beta^{a-h(v)}\right).
}\end{align*}
\DIFadd{Thus
}\begin{equation}\DIFadd{\label{eq:proof-bundle-minimum}
\min_{\sum_{v\in\Omega}f_v=\iota}
\sum_{v\in\Omega}\beta^{h(v)-a}f_v^2
=\kappa_{a,\Omega}\iota^2,
}\end{equation}
\DIFadd{and equality holds uniquely at
}\begin{equation}\DIFadd{\label{eq:proof-bundle-optimizer}
    f_v=\kappa_{a,\Omega}\beta^{a-h(v)}\iota.
}\end{equation}
\DIFadd{Applying \eqref{eq:proof-bundle-minimum} to \eqref{eq:proof-cut-identity} in every term of \eqref{eq:proof-energy-partition} yields the lower bound
}\begin{equation}\DIFadd{\label{eq:proof-value-lower-bound}
}{\DIFadd{\mathcal J}}\DIFadd{(w)\ge
\sum_{a=-\infty}^{h_{\max}}
\sum_{\Omega\in{\mathcal C}_a}
\kappa_{a,\Omega}(\xi_\Omega^\top w)^2.
}\end{equation}

\DIFadd{We next prove simultaneous attainability.  Prescribe
}\begin{equation}\DIFadd{\label{eq:proof-optimal-y}
    y^\star_{h(v)-a,v}
    =\kappa_{a,\Omega}\beta^{a-h(v)}
      \xi_\Omega^\top w,
    \qquad v\in\Omega.
}\end{equation}
\DIFadd{Fix $(a,\Omega)$ and let $E(\Omega)$ be the edges of the induced tree on $\Omega$.  For $e\in E(\Omega)$, set
}\DIFaddend \[
    \DIFdelbegin 
\DIFdelend \DIFaddbegin \DIFadd{q_e=v_{h(s(e))-a,e}.
}\DIFaddend \]
\DIFdelbegin \DIFdel{Summing this identity over $\Omega\in{\mathcal A}$ gives \eqref{eq:general-indicator-HC-identity}.
}

\DIFdel{Let $M=-S+\beta HC$.
Then \eqref{eq:prop-invariance-Phi} can be written as
}\DIFdel{
    \text{right-hand side of \eqref{eq:prop-invariance-Phi}}
    =I+\beta H-MD(H)^{-1}M^\top .
}
\DIFdelend \DIFaddbegin \DIFadd{At $v\in\Omega$, put $\tau=h(v)-a$ and define
}\[
\DIFadd{\eta_v=
\begin{cases}
w_v,&\tau=0,\\
y^\star_{\tau-1,v},&\tau\ge1,
\end{cases}
\qquad
y_v^{\rm out}=y^\star_{\tau,v}.
}\]
\DIFadd{Let $C_\Omega$ be the incidence matrix of the tree induced by $\Omega$.  The balance equation \eqref{eq:proof-balance-zero} or \eqref{eq:proof-balance-positive}, evaluated at vertex $v$ and time $\tau$, is
}\begin{equation}\DIFadd{\label{eq:proof-incidence-lift}
    C_\Omega q=y^{\rm out}-\eta,
}\end{equation}\DIFaddend 
\DIFdelbegin \DIFdel{This matrix is symmetric.  By \eqref{eq:boolean-atom-space}, a symmetric matrix belongs to ${\mathcal V}_{\mathcal B}$ if and only if it maps every vector in ${\mathcal W}_{\mathcal B}$ to zero.  To prove the equivalence, first suppose a symmetric matrix $X$ is constant on every block $A_\alpha\times A_\gamma$.  If $z\in{\mathcal W}_{\mathcal B}$, then each entry of $Xz$ is a linear combination of the sums $\sum_{i\in A_\alpha}z_i$, so $Xz=0$ by \eqref{eq:boolean-atom-contrast-set}.  Conversely, if $Xz=0$ for every $z\in{\mathcal W}_{\mathcal B}$, then $X(\hat e_i-\hat e_j)=0$ whenever $i$ and $j$ belong to the same class $A_\alpha$.  Thus the $i$th and $j$th columns of $X$ are equal.  Since $X$ is symmetric, the corresponding rows are equal, and $X$ is constant on every block $A_\alpha\times A_\gamma$.
Let $z\in{\mathcal W}_{\mathcal B}$.  Since $H\in{\mathcal V}_{\mathcal B}$, the equivalence just proved gives
}\DIFdel{
    Hz=0 .
}
\DIFdelend \DIFaddbegin \DIFadd{Equation \eqref{eq:proof-optimal-y} gives
}\DIFadd{\REV{\(\sum_{v\in\Omega}y_v^{\rm out}=\xi_\Omega^\top w\).}}\ \DIFaddend 
\DIFdelbegin \DIFdel{Lemma~\ref{lem:atom-contrast-lifting} gives $y\in\R^m$ such that
}\DIFdel{
    Sy=z,
    \qquad
    Cy\in{\mathcal W}_{\mathcal B}.
}
\DIFdelend \DIFaddbegin \DIFadd{Every vertex of $\Omega$ either has height $a$ or belongs to $V_{a+1}$.  The latter vertices are partitioned by the components $\Gamma\in{\mathcal C}_{a+1}$ contained in $\Omega$: if such a component meets $\Omega$, every path in $G[V_{a+1}]$ is also a path in $G[V_a]$, so the whole component lies in $\Omega$.  Therefore
}\begin{align*}
\DIFadd{\sum_{v\in\Omega}\eta_v
}&\DIFadd{=\sum_{\substack{v\in\Omega\\h(v)=a}}w_v
+\sum_{\substack{\Gamma\in{\mathcal C}_{a+1}\\\Gamma\subseteq\Omega}}
  \sum_{v\in\Gamma}y^\star_{h(v)-(a+1),v}}\\
&\DIFadd{=\sum_{\substack{v\in\Omega\\h(v)=a}}w_v
+\sum_{\substack{\Gamma\in{\mathcal C}_{a+1}\\\Gamma\subseteq\Omega}}
  \xi_\Gamma^\top w
=\xi_\Omega^\top w,
}\end{align*}\DIFaddend 
\DIFdelbegin \DIFdel{Applying the same equivalence to $H$ and
using \eqref{eq:prop-lift-y} gives
}\DIFdelend \DIFaddbegin \DIFadd{where the second equality uses \eqref{eq:proof-optimal-y}.
Hence $\ones^\top(y^{\rm out}-\eta)=0$.  The proof of Lemma~\ref{lem:tree-incidence-rank-Z}, applied to the connected tree on $\Omega$, gives
}\DIFadd{\REV{\(\range C_\Omega=\{b:\ones^\top b=0\}\)}}\ 
\DIFadd{and full column rank of $C_\Omega$.  Therefore \eqref{eq:proof-incidence-lift} has a unique solution.  Every transport variable belongs to exactly one equation of the form \eqref{eq:proof-incidence-lift}, because its level is $a=h(s(e))-t=h(t(e))-(t+1)$.  These solutions define a feasible sequence that attains equality in every instance of \eqref{eq:proof-bundle-minimum}.
}

\DIFadd{For $a\le0$, $V_a=V$ and ${\mathcal C}_a=\{V\}$.  Hence
}\[
    \DIFadd{\kappa_{a,V}
    =\frac{\beta^{-a}}{\sum_{v\in V}\beta^{-h(v)}}
}\]
\DIFadd{and
}\DIFaddend \begin{equation}\DIFdelbegin 
\DIFdel{HCy}\DIFdelend \DIFaddbegin \label{eq:proof-global-tail}
    \DIFadd{\sum_{a=-\infty}^{0}\kappa_{a,V}
    }\DIFaddend =\DIFdelbegin \DIFdel{0 }\DIFdelend \DIFaddbegin \DIFadd{\frac{1}{(1-\beta)\sum_{v\in V}\beta^{-h(v)}}=q_0}\DIFaddend .
\end{equation}
Equations \DIFdelbegin \DIFdel{\eqref{eq:prop-Hz-zero}}\DIFdelend \DIFaddbegin \DIFadd{\eqref{eq:proof-value-lower-bound}}\DIFaddend --\DIFdelbegin \DIFdel{\eqref{eq:prop-HCy-zero} give
}\DIFdelend \DIFaddbegin \DIFadd{\eqref{eq:proof-global-tail} prove
}\DIFaddend \begin{equation}\DIFdelbegin 
\DIFdel{M^\top z=-S^\top z,
    \qquad
    My=-z,
    \qquad
    D}\DIFdelend \DIFaddbegin \label{eq:proof-value-Hstar}
    {\DIFadd{\mathcal J}}\DIFaddend (\DIFdelbegin \DIFdel{H}\DIFdelend \DIFaddbegin \DIFadd{w}\DIFaddend )\DIFdelbegin \DIFdel{y}\DIFdelend =\DIFdelbegin \DIFdel{S}\DIFdelend \DIFaddbegin \DIFadd{w}\DIFaddend ^\top \DIFdelbegin \DIFdel{z .
}\DIFdelend \DIFaddbegin \DIFadd{H_\star w,
}\DIFaddend \end{equation}
\DIFdelbegin \DIFdel{Since $D(H)$ is nonsingular by hypothesis, \eqref{eq:prop-M-identities} gives $y=D(H)^{-1}S^\top z$.  Therefore \eqref{eq:prop-Phi-compact} gives
}\DIFdel{\begin{align*}
    \DIFdel{\bigl(I+\beta H-MD(H)^{-1}M^\top\bigr)z
    }&\DIFdel{=z+\beta Hz+MD(H)^{-1}S^\top z }\\
    &\DIFdel{=z+My=0,
}\end{align*}}
\DIFdel{where the second equality uses \eqref{eq:prop-Hz-zero} and $y=D(H)^{-1}S^\top z$, and the last equality uses \eqref{eq:prop-M-identities}.  Hence the }\DIFdelend \DIFaddbegin \DIFadd{where $H_\star$ is the }\DIFaddend right-hand side of \DIFdelbegin \DIFdel{\eqref{eq:prop-invariance-Phi} maps ${\mathcal W}_{\mathcal B}$ to zero.  The block-constant equivalence proved at the start of this proof gives \eqref{eq:prop-invariance-Phi}.
}
\DIFdelend \DIFaddbegin \DIFadd{\eqref{eq:explicit-height-solution}.  The unique minimizers in \eqref{eq:proof-bundle-optimizer} and the unique solutions of \eqref{eq:proof-incidence-lift} show that the optimal input sequence is unique.
}\DIFaddend 

\DIFdelbegin 
\DIFdelend \DIFaddbegin \DIFadd{Every coefficient in $H_\star$ is positive, so $H_\star\succeq0$.  Every input sequence in \eqref{eq:proof-reduced-value} consists of its first input $u$ followed by a tail starting from $r(w+Cu)$.  Conversely, every such first input and tail define an admissible sequence.  For every $\varphi\in\R^N$, \eqref{eq:proof-value-Hstar} gives ${\mathcal J}(r\varphi)=r^2\varphi^\top H_\star\varphi=\beta\varphi^\top H_\star\varphi$.  This decomposition gives
}\begin{equation}\DIFadd{\label{eq:proof-bellman}
w^\top H_\star w
=\inf_u\left\{
\|w-Su\|^2
+\beta(w+Cu)^\top H_\star(w+Cu)
\right\}.
}\end{equation}
\DIFaddend Let
\DIFdelbegin \DIFdel{$\Phi(H)$ denote the right-hand side of \eqref{eq:sr-reduced-H-DARE-expanded}.
Let
}
    \DIFdel{\Pi_{\mathcal B}
    =\sum_{\alpha=1}^{\ell}\frac{1}{|A_\alpha|}\xi_{A_\alpha}\xi_{A_\alpha}^\top .
}
\DIFdelend \DIFaddbegin \[
\DIFadd{D_\star=S^\top S+\beta C^\top H_\star C,
\qquad
M_\star=-S+\beta H_\star C.
}\]\DIFaddend 
\DIFdelbegin \DIFdel{This is the orthogonal projection onto the set of vectors that are constant on every set $A_\alpha$.  Choose $\alpha_0\ge 1/(1-\beta)$ and define
}
    \DIFdel{H_0=\alpha_0\Pi_{\mathcal B}, \qquad H_{j+1}=\Phi(H_j), \qquad j=0,1,2,\ldots .
}
\DIFdel{Since each atom $A_\alpha$ is an element of ${\mathcal B}$, \eqref{eq:boolean-atom-space} gives $H_0\in{\mathcal V}_{\mathcal B}$.  We next prove that $D(H_0)\succ0$.  By \eqref{eq:sr-DH}, for every $u\in\R^m$, }
    \DIFdel{u^\top D(H_0)u=\|Su\|^2+\alpha_0\beta\|\Pi_{\mathcal B}Cu\|^2 .
}
\DIFdel{If this quadratic form is zero, then $Su=0$ and $\Pi_{\mathcal B}Cu=0$.  The equality $Su=0$ says that the sum of the entries of $u$ over the edges leaving each vertex is zero.  The equality $\Pi_{\mathcal B}Cu=0$ says that }\DIFdelend \DIFaddbegin \DIFadd{At $w=0$, }\DIFaddend the \DIFdelbegin \DIFdel{sum of the entries of $Cu$ over every atom $A_\alpha$ is zero .  Each set ${\mathcal R}(v)$ in \eqref{eq:component-lattice} is a union of atoms by \eqref{eq:boolean-atom-relation}.  Hence $\xi_{{\mathcal R}(v)}^\top Cu=0$
for every $v\in V_{\rm rec}$.
By the definition of ${\mathcal R}(v)$ before \eqref{eq:component-lattice}, the only edges crossing the boundary of ${\mathcal R}(v)$ are the edges entering $v$.  Therefore $\xi_{{\mathcal R}(v)}^\top Cu=0$ says that the sum of the entries of }\DIFdelend \DIFaddbegin \DIFadd{objective in \eqref{eq:proof-bellman} is $u^\top D_\star u$.  If a nonzero }\DIFaddend $u$ \DIFdelbegin \DIFdel{over the edges entering $v$ is }\DIFdelend \DIFaddbegin \DIFadd{belonged to $\ker D_\star$, that input followed by the optimal tail from $rCu$ would be a second zero-cost optimal sequence from the }\DIFaddend zero \DIFdelbegin \DIFdel{.  Thus the sum of the entries of $u$ over the edges leaving each vertex is zero, }\DIFdelend \DIFaddbegin \DIFadd{initial state.  The unique minimizers in \eqref{eq:proof-bundle-optimizer} and the unique solutions of \eqref{eq:proof-incidence-lift} give a unique optimal input sequence, so this second sequence is impossible.  Hence $D_\star\succ0$.
}

\DIFadd{Expanding \eqref{eq:proof-bellman} gives
}\[
\DIFadd{w^\top(I+\beta H_\star)w
+2w^\top M_\star u
+u^\top D_\star u.
}\]
\DIFadd{Completing the square yields the unique minimizing feedback
}\begin{equation}\DIFadd{\label{eq:proof-Fstar}
    u=F_\star w,
    \qquad
    F_\star=-D_\star^{-1}M_\star^\top,
}\end{equation}
\DIFaddend and the \DIFdelbegin \DIFdel{sum of the entries of $u$ over the edgesentering each receiving vertex is zero.  A leaf has only one incident edge }\DIFdelend \DIFaddbegin \DIFadd{DARE identity
}\begin{equation}\DIFadd{\label{eq:proof-Hstar-DARE}
    H_\star
    =I+\beta H_\star-M_\star D_\star^{-1}M_\star^\top.
}\end{equation}

\DIFadd{It remains to prove stability.  Let
}\DIFadd{\REV{\(A_\star=r(I+CF_\star)\).}}\ 
\DIFadd{Substituting \eqref{eq:proof-Fstar} into \eqref{eq:proof-bellman} gives
}\begin{equation}\DIFadd{\label{eq:proof-Hstar-Lyapunov}
H_\star
=(I-SF_\star)^\top(I-SF_\star)
+A_\star^\top H_\star A_\star.
}\end{equation}
\DIFadd{Suppose $A_\star\psi=\lambda\psi$ for a nonzero complex vector $\psi$ with $|\lambda|\ge1$.  With $^*$ denoting conjugate transpose, the Hermitian quadratic form of \eqref{eq:proof-Hstar-Lyapunov} gives
}\[
\DIFadd{(1-|\lambda|^2)\psi^*H_\star\psi
=\|(I-SF_\star)\psi\|^2.
}\]
\DIFadd{The left side is nonpositive and the right side is nonnegative, so both sides are zero.  Set $u=F_\star\psi$.  Then $\psi=Su$, and
}\DIFadd{\REV{\(\lambda\psi=A_\star\psi=r(\psi+Cu)=rRu\).}}\ 
\DIFadd{Thus $(rR-\lambda S)u=0$.  Lemma~\ref{lem:tree-edge-column-rank}, with $z=\lambda$, gives $u=0$, and then $\psi=Su=0$, a contradiction.  Therefore $\rho(A_\star)<1$.  Lemma~\ref{lem:reduced-dare-existence-conditions} now gives $H=H_\star$}\DIFaddend .
\DIFdelbegin \DIFdel{If this edge leaves the leaf, the outgoing-sum condition gives that its entry in $u$ is zero.  If this edge enters the leaf, the incoming-sum condition gives that its entry in $u$ is zero.  After removing that leaf and edge, the two zero-sum conditions are unchanged at all remaining vertices, because the removed edge has entry zero.  Induction over the number of edges gives $u=0$.  Hence $D(H_0)\succ0$.
Proposition~\ref{prop:intersection-span-invariance} gives
$\Phi(H_0)\in{\mathcal V}_{\mathcal B}$.  Since the last term in \eqref{eq:prop-invariance-Phi} is positive semidefinite,
}\DIFdelend \DIFaddbegin 

\DIFadd{\REV{Positivity of the coefficients in \eqref{eq:explicit-height-solution} follows from $0<\beta<1$ and nonemptiness of every height component.}}\end{proof}

\begin{proof}[Proof of Corollary~\ref{cor:exact-incidence-compression}]
\DIFadd{For edge $e_i:s_i\to t_i$, \eqref{eq:tree-height} gives $h(t_i)=h(s_i)+1=a_i$.  At level $a_i$, the receiver belongs to $V_{a_i}$ and the sender does not, so
}\DIFadd{\REV{\(\xi_{\Omega_i}^\top c_i=1\).}}\ 
\DIFadd{If $a<a_i$, both endpoints belong to $V_a$ and their edge places them in the same component of $G[V_a]$.  If $a>a_i$, neither endpoint belongs to $V_a$.  Consequently, among all positive-level terms in \eqref{eq:explicit-height-solution}, only the term indexed by $(a_i,\Omega_i)$ has a nonzero product with $c_i$.  The global term vanishes because $\ones^\top c_i=0$.  This proves \eqref{eq:exact-HC-column}.
}

\DIFadd{Applying $c_j^\top$ to \eqref{eq:exact-HC-column} gives
}\DIFaddend \[
    \DIFdelbegin \DIFdel{\Phi(H_0)\preceq I+\beta H_0}\DIFdelend \DIFaddbegin \DIFadd{c_j^\top Hc_i}\DIFaddend =\DIFdelbegin \DIFdel{I+\alpha_0\beta\Pi_{\mathcal B} }\DIFdelend \DIFaddbegin \DIFadd{\kappa_i c_j^\top\xi_{\Omega_i}}\DIFaddend .
\]
Let \DIFdelbegin \DIFdel{$x\in\R^N$ and write $x=x_c+x_w$, where $x_c=\Pi_{\mathcal B}x$ and $x_w=x-x_c$.  Then $x_w\in{\mathcal W}_{\mathcal B}$ by \eqref{eq:boolean-atom-contrast-set}.  The block-constant equivalence proved in Proposition~\ref{prop:intersection-span-invariance} gives $\Phi(H_0)x_w=0$ and $H_0x_w=0$.  Therefore
}\DIFdel{\begin{align*}
    \DIFdel{x^\top\Phi(H_0)x
    }&\DIFdel{=x_c^\top\Phi(H_0)x_c \le (1+\alpha_0\beta)\|x_c\|^2 }\\
    &\DIFdel{\le \alpha_0\|x_c\|^2
      =x^\top H_0x,
}\end{align*}}
\DIFdel{where the second inequality uses $\alpha_0\ge1/(1-\beta)$.  Hence $\Phi(H_0)\preceq H_0$.
The monotone Riccati iteration initialized at the positive semidefinite supersolution $H_0$ converges to the stabilizing solution; see the DARE convergence theorem in~\cite{BittantiLaubWillems1991}.  Hence $H_j\to H$.  The same theorem gives $D(H_j)\succ0$ for every $j$.  Proposition~\ref{prop:intersection-span-invariance} and induction give $H_j\in{\mathcal V}_{\mathcal B}$ for every $j$.  Since ${\mathcal V}_{\mathcal B}$ is the finite-dimensional linear span in \eqref{eq:boolean-atom-space}, it is closed.  Passing to the limit $H_j\to H$ gives $H\in{\mathcal V}_{\mathcal B}$}\DIFdelend \DIFaddbegin \DIFadd{edge $e_j$ have sender $s_j$ and receiver $t_j$.  If $(a_j,\Omega_j)=(a_i,\Omega_i)$, then $t_j\in\Omega_i$ and $s_j\notin\Omega_i$, so $c_j^\top\xi_{\Omega_i}=1$.  If $a_j<a_i$, both endpoints of $e_j$ are outside $V_{a_i}$.  If $a_j>a_i$, both endpoints belong to $V_{a_i}$ and their edge places them in the same component of $G[V_{a_i}]$.  If $a_j=a_i$ but $\Omega_j\ne\Omega_i$, neither endpoint belongs to $\Omega_i$.  Hence $c_j^\top\xi_{\Omega_i}=0$ in all remaining cases.  This proves \eqref{eq:exact-CtHC-entry}.  Grouping equal pairs $(a_i,\Omega_i)$ gives \eqref{eq:exact-CtHC-blocks}}\DIFaddend .
\end{proof}

\begin{proof}[Proof of Lemma~\ref{lem:transitive-closure-classes}]
The \DIFdelbegin \DIFdel{chain characterization of $P^*$ is the standard characterization of transitive closure for a binary relation~\cite[Sec.~9.4]{Rosen2012}.  Since $P(p,p)=1$, the one-index chain gives $P^*(p,p)=1$.  Since $P$ is symmetric, reversing any }\DIFdelend \DIFaddbegin \DIFadd{relation $\Pi^*(p,q)=1$ holds exactly when there is a }\DIFaddend chain \DIFaddbegin \DIFadd{$i_0=p,i_1,\ldots,i_\ell=q$ with $\Pi(i_{j-1},i_j)=1$ for every $j$.  The diagonal condition gives reflexivity.  Symmetry of $\Pi$ allows every chain to be reversed, which gives symmetry.  Concatenating a chain }\DIFaddend from $p$ to $q$ \DIFdelbegin \DIFdel{gives }\DIFdelend \DIFaddbegin \DIFadd{with }\DIFaddend a chain from $q$ to \DIFdelbegin \DIFdel{$p$, so $P^*(p,q)=1$ gives $P^*(q,p)=1$.  If $P^*(p,q)=1$ and $P^*(q,r)=1$, concatenating the two chains gives $P^*(p,r)=1$.  }\DIFdelend \DIFaddbegin \DIFadd{$\nu$ gives transitivity.  }\DIFaddend Thus $\sim$ is an equivalence relation.
\DIFaddbegin 

\DIFaddend If $p$ and $q$ are in different equivalence classes, then \DIFdelbegin \DIFdel{$P^*(p,q)=0$.  Since $P(p,q)=1$ would give a chain of length one and hence $P^*(p,q)=1$, we must have $P(p,q)=0$.  The hypothesis on $A$ gives $A_{pq}=0$.  Therefore entries connecting different equivalence classes vanish.  After reordering }\DIFdelend \DIFaddbegin \DIFadd{$\Pi^*(p,q)=0$.  The implication $\Pi(p,q)=1\Rightarrow \Pi^*(p,q)=1$ gives $\Pi(p,q)=0$, and the hypothesis on $X$ gives $X_{pq}=0$.  Reordering }\DIFaddend the indices so \DIFdelbegin \DIFdel{all indices in each class appear consecutively, $A$ is }\DIFdelend \DIFaddbegin \DIFadd{that each equivalence class is consecutive therefore makes $X$ }\DIFaddend block diagonal.
\end{proof}

\begin{proof}[Proof of Proposition~\ref{prop:general-separator-FZ-sparsity}]
\DIFdelbegin \DIFdel{Write $H=\sum_{\Omega\in{\mathcal L}}q_\Omega\xi_\Omega\xi_\Omega^\top$ and $b_\Omega=C^\top\xi_\Omega$.  Then \(D=S^\top S+\beta\sum_{\Omega\in{\mathcal L}}q_\Omega b_\Omega b_\Omega^\top\).  Lemma~\ref{lem:reduced-dare-existence-conditions} gives that $D$ is nonsingular.
By the definition of $S$ in Section~\ref{sec:problem}, the $(p,q)$ entry of $S^\top S$ }\DIFdelend \DIFaddbegin \DIFadd{Equation \eqref{eq:sr-DH} defines $D=S^\top S+\beta C^\top HC$.  The entry $(S^\top S)_{pq}$ }\DIFaddend is zero unless \DIFdelbegin \DIFdel{$e_p$ and $e_q$ have the same sender.  Lemma~\ref{lem:downstream-compression} shows that the $(p,q)$ entry of $b_\Omega b_\Omega^\top$ }\DIFdelend \DIFaddbegin \DIFadd{$s_p=s_q$, and \eqref{eq:exact-CtHC-entry} shows that $(C^\top HC)_{pq}$ }\DIFaddend is zero unless \DIFdelbegin \DIFdel{$e_p$ and $e_q$ both cross the boundary of $\Omega$.  Therefore }\DIFdelend \DIFaddbegin \DIFadd{$(a_p,\Omega_p)=(a_q,\Omega_q)$.  Hence }\DIFaddend $D_{pq}=0$ whenever $P_D(p,q)=0$.  \DIFdelbegin \DIFdel{The matrix $P_D$ is symmetric and satisfies $P_D(p,p)=1$ for every $p$, so }\DIFdelend Lemma~\ref{lem:transitive-closure-classes} \DIFdelbegin \DIFdel{applies with $P=P_D$ and $A=D$.  Thus, after reordering the indices by the }\DIFdelend \DIFaddbegin \DIFadd{makes $D$ block diagonal over the }\DIFaddend equivalence classes of $P_D^*$\DIFdelbegin \DIFdel{, $D=\operatorname{diag}(D_1,\ldots,D_t)$.  Since }\DIFdelend \DIFaddbegin \DIFadd{.  The matrix }\DIFaddend $D$ is nonsingular \DIFdelbegin \DIFdel{, each $D_a$ is nonsingular, and $D^{-1}=\operatorname{diag}(D_1^{-1},\ldots,D_t^{-1})$.
Now put $G=(S^\top-\beta C^\top H)Z$.  }\DIFdelend \DIFaddbegin \DIFadd{by Lemma~\ref{lem:reduced-dare-existence-conditions}, so $D^{-1}$ has the same block decomposition.
}

\DIFadd{Let \REV{\(G_0=(S^\top-\beta C^\top H)Z\).}}\ 
\DIFaddend For the first $m$ columns of $Z$\DIFdelbegin \DIFdel{from \eqref{eq:Z-definition}}\DIFdelend , the sender term \DIFdelbegin \DIFdel{$(S^\top C)_{ab}$ }\DIFdelend \DIFaddbegin \DIFadd{$(S^\top C)_{pq}$ }\DIFaddend can be nonzero only when \DIFdelbegin \DIFdel{$C_{s(a),b}\ne0$.  The $H$ term satisfies $(C^\top HC)_{ab}$, which is zero unless some $\Omega\in{\mathcal L}$ separates the endpoints represented by columns $a$ and $b$ of $C$, by Lemma~\ref{lem:downstream-compression}}\DIFdelend \DIFaddbegin \DIFadd{$C_{s_p,q}\ne0$.  Equation \eqref{eq:exact-CtHC-entry} shows that $(C^\top HC)_{pq}$ can be nonzero only when $(a_p,\Omega_p)=(a_q,\Omega_q)$}\DIFaddend .  Thus the edge-column support of \DIFdelbegin \DIFdel{$G$ }\DIFdelend \DIFaddbegin \DIFadd{$G_0$ }\DIFaddend is contained in $P_G$.
\DIFdelbegin \DIFdel{By \eqref{eq:sr-known-H-F} , $FZ$ has the same zero pattern as $D^{-1}G$ up to the nonzero scalar $r$.  For }\DIFdelend \DIFaddbegin 

\DIFadd{Equation \eqref{eq:sr-known-H-F} gives $FZ=rD^{-1}G_0$.  If $(FZ)_{pq}\ne0$ for }\DIFaddend $q\le m$, \DIFdelbegin \DIFdel{a nonzero entry $(FZ)_{pq}$ therefore requires }\DIFdelend \DIFaddbegin \DIFadd{there is }\DIFaddend an index $k$ \DIFdelbegin \DIFdel{with }\DIFdelend \DIFaddbegin \DIFadd{such that }\DIFaddend $(D^{-1})_{pk}\ne0$ and \DIFdelbegin \DIFdel{$G_{kq}\ne0$}\DIFdelend \DIFaddbegin \DIFadd{$(G_0)_{kq}\ne0$}\DIFaddend .  The block \DIFdelbegin \DIFdel{diagonal form }\DIFdelend \DIFaddbegin \DIFadd{decomposition }\DIFaddend of $D^{-1}$ gives \DIFdelbegin \DIFdel{$(D^{-1})_{pk}=0$ unless $p$ and $k$ are in the same equivalence class, which is }\DIFdelend $P_D^*(p,k)=1$\DIFdelbegin \DIFdel{by the definition of $P_D^*$ before Lemma~\ref{lem:transitive-closure-classes}.  Also $G_{kq}\ne0$ implies }\DIFdelend \DIFaddbegin \DIFadd{, while the support of $G_0$ gives }\DIFaddend $P_G(k,q)=1$.  Hence $(P_D^*P_G)(p,q)=1$ in Boolean \DIFdelbegin \DIFdel{matrix arithmeticwhenever $(FZ)_{pq}\ne0$ with $q\le m$, proving }\DIFdelend \DIFaddbegin \DIFadd{arithmetic.  This proves }\DIFaddend \eqref{eq:separator-FZ-mask}.  The \DIFdelbegin \DIFdel{last column of $Z$ is the mean column }\DIFdelend \DIFaddbegin \DIFadd{calculation does not impose a zero in the column corresponding to }\DIFaddend $\ones$\DIFdelbegin \DIFdel{, and no zero claim is made for it}\DIFdelend .
\end{proof}

\DIFdelbegin 
\DIFdel{By \eqref{eq:outward-component-lattice}, ${\mathcal L}=\{T_v:v\in V\}$.  Since $H\in{\mathcal V}_{\mathcal L}$, \eqref{eq:intersection-component-space} gives
$H\in\operatorname{span}\{\xi_{T_v}\xi_{T_v}^\top:v\in V\}$.
Definition~\ref{def:rooted-tree-notation} gives $\xi_{T_v}=\chi_v$ for every $v\in V$.  Substituting this identity into the displayed span gives \eqref{eq:outward-subtree-span}.
}

\DIFdelend \begin{proof}[Proof of Corollary~\ref{cor:outward-rooted-FZ-support}]
\DIFdelbegin \DIFdel{Using \eqref{eq:outward-component-lattice} gives ${\mathcal L}=\{T_v:v\in V\}$ under Assumption~\ref{ass:outward-rooted}.  By }\DIFdelend Lemma~\ref{lem:outward-rooted-orientation} \DIFdelbegin \DIFdel{, every edge is directed }\DIFdelend \DIFaddbegin \DIFadd{directs every edge }\DIFaddend away from the root.  \DIFdelbegin \DIFdel{Thus, if edge $e_i$ has sender $p_i$, then it has receiver $h_i$, where $h_i$ is a child of $p_i$, and $e_i=(p_i,h_i)$.  The only subtree }\DIFdelend \DIFaddbegin \DIFadd{Equation \eqref{eq:tree-height} therefore makes $h(v)$ equal to the number of edges on the path from the root to $v$.  For each non-root vertex $v$, the component of $G[V_{h(v)}]$ containing $v$ is }\DIFaddend $T_v$\DIFdelbegin \DIFdel{whose boundary is crossed by $e_i$ }\DIFdelend \DIFaddbegin \DIFadd{.  Substitution into \eqref{eq:explicit-height-solution} gives \eqref{eq:outward-subtree-span}.
}

\DIFadd{For edge $e_i:s_i\to t_i$, the pair $(a_i,\Omega_i)$ }\DIFaddend is \DIFdelbegin \DIFdel{$T_{h_i}$.  Hence two edges cross the boundary of the same set in ${\mathcal L}$ only when they are the same edge .  }\DIFdelend \DIFaddbegin \DIFadd{$(h(t_i),T_{t_i})$.  Distinct edges have distinct receivers and hence distinct pairs.  Equation \eqref{eq:exact-CtHC-entry} therefore makes $C^\top HC$ diagonal.
}

\DIFaddend The definition of $P_D$ \DIFdelbegin \DIFdel{before Lemma~\ref{lem:transitive-closure-classes} therefore gives
\(P_D(i,j)=1 \quad\Longleftrightarrow\quad i=j\text{ or }p_i=p_j\).
}\DIFdelend \DIFaddbegin \DIFadd{now gives
}\[
    \DIFadd{P_D(i,j)=1
    \quad\Longleftrightarrow\quad
    i=j\ \text{or}\ s_i=s_j.
}\]
\DIFaddend Thus the equivalence classes of $P_D^*$ are the sets $I_v$.  \DIFdelbegin \DIFdel{For }\DIFdelend \DIFaddbegin \DIFadd{In row $i$ of }\DIFaddend $P_G$, the \DIFdelbegin \DIFdel{boundary-crossing part contributes only the diagonal }\DIFdelend \DIFaddbegin \DIFadd{condition $(a_i,\Omega_i)=(a_q,\Omega_q)$ contributes only }\DIFaddend column $i$.  The \DIFdelbegin \DIFdel{sender part in the definition of $P_G$ is nonzero exactly when the $q$th column of $C$ is nonzero at $p_i$.  Under Assumption~\ref{ass:outward-rooted}, this happens for columns corresponding to edges in $I_{p_i}$ and, if $p_i\ne r$, for the column $i^-(p_i)$.  Therefore row $i$ of $P_G$ has support contained in the columns in \eqref{eq:outward-row-mask-U-bound}.  }\DIFdelend \DIFaddbegin \DIFadd{condition $C_{s_i,q}\ne0$ contributes the columns for edges in $I_{s_i}$ and, when $s_i\ne v_{\rm root}$, the column $i^-(s_i)$.  }\DIFaddend Multiplication by $P_D^*$ \DIFdelbegin \DIFdel{in \eqref{eq:separator-FZ-mask} only mixes rows with the same sender , so row $i$ of $P_D^*P_G$ has }\DIFdelend \DIFaddbegin \DIFadd{mixes only rows having sender $s_i$, so }\DIFaddend the same column containment \DIFaddbegin \DIFadd{holds for row $i$ of $P_D^*P_G$}\DIFaddend .  Proposition~\ref{prop:general-separator-FZ-sparsity} gives \DIFdelbegin \DIFdel{the stated support for the first $m$ columns of $FZ$ and makes no zero claim for the mean column}\DIFdelend \DIFaddbegin \DIFadd{\eqref{eq:outward-row-mask-U-bound}}\DIFaddend .
\end{proof}
\endgroup
\begingroup\color{black}
\begin{proof}[Proof of Proposition~\ref{prop:weighted-retained-zeros}]
Let $P_0,P_W$ be the optimal value matrices and put
$\Delta H=H_W-H_0$. In the weighted DARE, $W+B^\top P_WB$
replaces $B^\top P_WB$ in \eqref{eq:dare}. Thus the block calculation
in \eqref{eq:sr-P-from-H}, with $W+D(P_W)$ replacing $D(P_W)$ in
\eqref{eq:sr-H-of-P}, gives
$P_W=Q+\beta[I\ R]^\top H_W[I\ R]$.
By \eqref{eq:sr-dynamics}, this gives $W+B^\top P_WB=D_W$ and
$B^\top P_WA=r(-S^\top+\beta C^\top H_W)[I\ R]$.
Subtracting \eqref{eq:sr-P-from-H} and using \eqref{eq:sr-DH},
\eqref{eq:weighted-denominator}, and the feedback formula following
\eqref{eq:dare} gives
\begin{equation}\label{eq:weighted-dare-identities}
\begin{gathered}
P_W-P_0=\beta[I\ R]^\top\Delta H[I\ R],\\
D_W=D_0+W+\beta C^\top\Delta HC,\\
D_WF_W=r(S^\top-\beta C^\top H_W).
\end{gathered}
\end{equation}
The added nonnegative input cost gives $P_W-P_0\succeq0$.
Taking $x=(v,0)$ in \eqref{eq:weighted-dare-identities} yields
$\Delta H\succeq0$. Lemma~\ref{lem:reduced-dare-existence-conditions}
and \eqref{eq:sr-DH} give $D_0\succ0$, so
\eqref{eq:weighted-dare-identities} gives $D_W\succ0$.

For $v\in\mathcal V_W$, Cayley--Hamilton and
\eqref{eq:weighted-invariant-subspace} give $WF_0T_0^kv=0$ for all
$k\ge0$. The unweighted optimal trajectory from $(v,0)$ therefore
attains the weighted optimum as well. Since $D_W\succ0$ makes its
first input unique, $F_Wv=F_0v$. Take $v=z_j$.

For the row claim, put $a_i=D_0^{-1}\hat e_i$. From
\eqref{eq:weighted-row-test}, $Wa_i=0$ and
$Ca_i\in\mathcal V_W$. The first identity of
\eqref{eq:weighted-dare-identities} and the preceding cost equality
give $(Ca_i)^\top\Delta H(Ca_i)=0$. As $\Delta H\succeq0$,
$\Delta HCa_i=0$. Thus $D_Wa_i=\hat e_i$. Subtracting
$D_WF_0$ from the last identity of
\eqref{eq:weighted-dare-identities} and using
\eqref{eq:sr-known-H-F} and the second identity of
\eqref{eq:weighted-dare-identities} gives
\begin{equation}\label{eq:weighted-gain-difference}
\begin{aligned}
D_W(F_W-F_0)
&=-WF_0-\beta C^\top\Delta H(rI+CF_0)\\
&=-WF_0-\beta C^\top\Delta H T_0.
\end{aligned}
\end{equation}
The right-hand side of \eqref{eq:weighted-gain-difference} vanishes
after left multiplication by $a_i^\top$, since $Wa_i=0$ and
$\Delta HCa_i=0$. Since $D_W$ is symmetric and
$D_Wa_i=\hat e_i$, this gives
$\hat e_i^\top(F_W-F_0)=0$. Multiplication by $Z$ proves the row claim.
\end{proof}
\endgroup
\DIFaddbegin 

\DIFaddend %
\section*{Declaration of generative artificial intelligence and artificial-intelligence-assisted technologies in the manuscript preparation process}
During the preparation of this work \REV{ChatGPT} \REV{was} used for proof exploration, initial drafting, and adversarial review. The authors have improved and edited the content as needed and \REV{take} full responsibility for the content of the published article.

\DIFaddbegin
\section*{Acknowledgment}
\DIFadd{The \REV{authors thank} Richard Pates for useful comments on an early draft of this paper.}
\DIFaddend

\bibliographystyle{IEEEtran}
\bibliography{references}

\end{document}